\documentclass[12pt]{amsart}
\usepackage{indentfirst, latexsym, bm, amsmath, eufrak, amsthm, amssymb}
\usepackage[colorlinks,
linkcolor=black,citecolor=black]{hyperref}
\usepackage{upgreek}
\usepackage{bm}
\usepackage{url}
\usepackage{lineno}

\theoremstyle{definition}
\theoremstyle{remark}
\numberwithin{equation}{section}
\newtheorem{theorem}{Theorem}[section]
\newtheorem*{theorem*}{Theorem}

\newtheorem{lemma}{Lemma}[section]
\newtheorem*{lemma*}{Lemma}
\newtheorem{corollary}{Corollary}[section]
\newtheorem{remark}{Remark}[section]
\newtheorem{proposition}{Proposition}[section]
\newtheorem{definition}{Definition}[section]
\newtheorem{example}{Example}[section]
\begin{document}
\title[Hamilton-Tian Conjecture]{On the Hamilton-Tian Conjecture in a compact transverse Fano Sasakian $5$-manifold}
\author{$^{\ast}$Shu-Cheng Chang}
\address{Shanghai Institute of Mathematics and Interdisciplinary Sciences (SIMIS), Shanghai, 200433, P.R. China}
\email{scchang@math.ntu.edu.tw }
\author{$^{\dag}$Yingbo Han}
\address{School of Mathematics and Statistics, Xinyang Normal University, Xinyang,464000, Henan, P. R. China}
\email{yingbohan@163.com}
\author{$^{\ast \ast}$Chien Lin}
\address{Department of Mathematics, National Taiwan Normal University, Taipei, Taiwan}
\email{chienlin@ntnu.edu.tw}
\author{$^{^{\ddag}}$Chin-Tung Wu}
\address{Department of Applied Mathematics, National Pingtung University, Pingtung 90003, Taiwan}
\email{ctwu@mail.nptu.edu.tw }
\thanks{$^{\ast}$Supported in part by Startup Foundation
for Advanced Talents of the Shanghai Institute for Mathematics and
Interdisciplinary Sciences (No.2302-SRFP-2024-0049).$^{\dag}$Supported in part by an NSFC grant No. 12571056 and NSF of Henan Province
No. 252300421497. $^{\ast \ast}$Supported in part by NSTC 114-2115-M-003-002-MY2, Taiwan. $^{\ddag}$Supported in part by NSTC 114-2115-M-153-001, Taiwan. }
\maketitle

\begin{abstract}
In this paper, we first confirm the Hamilton-Tian conjecture for the
Sasaki-Ricci flow in a compact transverse Fano quasi-regular Sasakian
$5$-manifold with klt\textbf{\ }foliation singularities. Secondly, we derive
the compactness theorem of Sasaki-Ricci solitons on transverse Fano quasi-regular Sasakian
$5$-manifolds. Then, by the second Sasakian structure theorem, we confirm the Hamilton-Tian
conjecture for a compact transverse Fano Sasakian $5$-manifold. With its
applications, we show that the gradient Sasaki-Ricci soliton orbifold metric
on a compact Sasakian $5$-manifold is Sasaki-Einstein if $M$ is transverse $K$-stable.
\end{abstract}
\tableofcontents

\section{Introduction}
Sasakian geometry is the odd-dimensional analogue of K\"{a}hler geometry. In
particular, $5$-dimensional Sasaki--Einstein manifolds are closely related to
Calabi--Yau $3$-folds via the K\"{a}hler cone construction. This
correspondence makes Sasakian geometry an important tool in the study of
Calabi--Yau geometry and String theory. Thus, the geometrization problem on
Sasakian $5$-manifolds not only contributes to the classification of these
manifolds but also enhances our understanding of Calabi--Yau geometry with
deep connections to the AdS/CFT correspondence. On the other hand, the class
of simply connected, closed, oriented, smooth $5$-manifolds is classifiable
under diffeomorphism due to Smale-Barden (\cite{s}, \cite{b}).

Let $(M,\eta,\xi,\Phi,g)$ be a compact Sasakian manifold of dimension $2n+1.$
If the orbits of the Reeb vector field $\xi$ are all closed and hence circles,
then integrates to an isometric $U(1)$ action on $(M,g)$. Since it is nowhere
zero this action is locally free, that is, the isotropy group of every point
in $S$ is finite. If the $U(1$) action is in fact free then the Sasakian
structure is said to be regular. Otherwise, it is said to be quasi-regular. If
the orbits of $\xi$ are not all closed,\textbf{\ }the Sasakian structure is
said to be irregular. However, by the second structure theorem (\cite{ru},
\cite{bg}), any Sasakian structure $(\xi,\eta,\Phi)$ on $(M,g)$ is either
quasi-regular or there is a sequence of quasi-regular Sasakian structures
$(M,\xi_{i},\eta_{i},\Phi_{i},g_{i})$ converging in the compact-open
$C^{\infty}$-topology to $(\xi,\eta,\Phi,g).$ It means that there always
exists a quasi-regular Sasakian structure $(\xi,\eta,\Phi)$ on $(M,g)$.

Significant progress has been made in Sasaki--Einstein geometry by
Boyer--Galicki, Koll\'{a}r, Collins--Sz\'{e}kelyhidi, and Ono--Futaki--Wang,
etc. Furthermore, Kobayashi and Boyer--Galicki established a close
correspondence between quasi-regular Sasakian $5$-manifolds and log del Pezzo
surfaces, relating Sasaki--Einstein and K\"{a}hler--Einstein structures
respectively, thereby illustrating the strong interplay between Sasakian and
complex geometry.

By applying the Seifert $C^{\ast}$-bundle construction, there are questions
asked by Koll\'{a}r (\cite{k3}):
\begin{enumerate}
\item Describe all Sasakian manifolds $M$ which have a given Seifert
$\mathbb{S}^{1}$\textbf{-}bundle structure$\pi:M\rightarrow(Z,\Delta)$ over
an orbifold with positive orbifold Chern class?

\item Given a Sasakian manifold $M$\textbf{, }describe all Seifert $S^{1}%
$-bundle structures $\pi:M\rightarrow(Z,\Delta)$ over an orbifold with
positive orbifold Chern class?
\end{enumerate}
We refer to the book by Boyer--Galicki (\cite{bg}) and references therein in
some details for the related subjects. It is our goal to apply
\textbf{foliation analytic MMP} as in section $2$ to attach the above
geometrization problems of compact quasi-regular Sasakian $5$-manifolds.

In a series of our previous articles (\cite{chlw1}, \cite{chlw2},
\cite{ccllw}, \cite{cclw}), we address the related issues on the
geometrization and classification problems of a compact quasi-regular Sasakian
$5$-manifold. In this paper, we will deal with problems in a general Sasakian
$5$-manifold without the property of quasi-regular. In fact, we will answer
the Hamilton-Tian conjecture for the Sasaki-Ricci flow in a compact transverse
Fano Sasakian $5$-manifold. We refer to the topic $(1)$ of the conjecture
picture as in subsection $2.2$.

It is served as a generalization of the well-known Hamilton-Tian conjecture
for the K\"{a}hler-Ricci flow ( \cite{tz1}, \cite{csw}, \cite{bbegz},
\cite{bam}, \cite{cw}, \cite{tzzzz}). For more details, we refer to references therein.

More precisely, we confirm the Hamilton-Tian conjecture which states that the
solution of Sasaki-Ricci flow (\ref{2025}) converges (along a subsequence) to
a shrinking Sasaki-Ricci soliton with mild singularities in a compact
transverse Fano Sasakian manifold of dimension five.

As a consequence of the Hamilton-Tian conjecture, we have the existence
theorem of Sasaki-Ricci solitons and Sasaki-Einstein in a compact transverse
Fano Sasakian $5$-manifold with klt foliation singularities as in Theorem
\ref{T99} and Corollary \ref{C12}.

We first assume that $M$ is a compact quasi-regular transverse Fano
$5$-Sasakian manifold with Type I foliation singularities (Definition
\ref{D21}) in the sense that the space $Z$ of leaves is a del Pezzo orbifold
surface of cyclic quotient klt singularities (well-formed) which means its
orbifold singular locus and algebro-geometric singular locus coincide,
equivalently $Z$ has no branch divisors. Then we have

\begin{proposition}
\label{T66} Let $(M,\xi,\eta_{0},g_{0}^{T})$ be a compact quasi-regular
transverse Fano Sasakian manifold of dimension five and $(Z_{0},h_{0}%
,\omega_{h_{0}})$ denote the space of leaves of the characteristic foliation
which is a del Pezzo orbifold surface of cyclic quotient klt singularities
with codimension two orbifold singularities $\Sigma_{0}$. Then, by a
$D$-homothetic deformation, under the normalized Sasaki-Ricci flow
\begin{equation}
\left \{
\begin{array}
[c]{lcl}%
\frac{\partial}{\partial t}\omega(t) & = & -Ric^{T}(\omega(t))+\omega(t),\\
\omega(0) & = & \omega_{0}.
\end{array}
\right.  \label{2025}%
\end{equation}

\begin{enumerate}
\item $(M(t),\xi,\eta(t),g^{T}(t))$ converges to a compact quasi-regular
transverse Fano Sasaki-Ricci soliton $(M_{\infty},\xi,\eta_{\infty},g_{\infty
}^{T}).$

\item The leave space of a del Pezzo orbifold surface $(Z_{\infty},h_{\infty
})$ can have at worst finite isolated klt singularities $\Sigma_{\infty
}=\{x_{1},\cdots,x_{N}\}$ with $g_{\infty}^{T}=\pi^{\ast}(h_{\infty})$ such
that $h_{\infty}$ is the smooth gradient K\"{a}hler-Ricci soliton metric in
the Cheeger-Gromov topology on $Z_{\infty}\backslash \Sigma_{\infty}$.

\item $M_{\infty}$ is at worst foliation log terminal for the singular set of
$\mathbb{S}^{1}$-fibres of $\{ \xi_{x_{1}},\cdots,\xi_{x_{N}}\}$.
\end{enumerate}
\end{proposition}
Now by applying the methods as in Guo-Phong-Song-Sturm \cite{gpss}, we have
the compactness theorem of Sasaki-Ricci solitons on transverse Fano quasi-regular
Sasakian manifolds of dimension five as in Theorem \ref{T55}. And then
by the second structure theorem (Proposition \ref{P22}), we confirm the
Hamilton-Tian conjecture for the Sasaki-Ricci flow on a compact transverse
Fano Sasakian manifold of dimension five by removing the condition of
quasi-regular:
\begin{theorem}
\label{T99} Let $(M,\xi,\eta_{0},g_{0}^{T})$ be a compact transverse Fano
Sasakian manifold of dimension five with Type I foliation singularities. Then,
under the Sasaki-Ricci flow (\ref{2025})

\begin{enumerate}
\item $(M(t),\xi,\eta(t),g^{T}(t))$ converges to a compact gradient
Sasaki-Ricci orbifold soliton $(M_{\infty},\xi,\eta_{\infty},g_{\infty}^{T}).$

\item $g^{T}(t)$ converges to a smooth Sasaki-Ricci soliton metric in the
Cheeger-Gromov topology on $M_{\infty}\backslash \Sigma_{\infty}^{T}$. Here
$\Sigma_{\infty}^{T}$ is the singular set of $\mathbb{S}^{1}$-fibres of $\{
\xi_{p_{1}},\cdots,\xi_{p_{N}}\}$.

\item $M_{\infty}$ is at worst foliation log terminal at the singular set of
$\  \mathbb{S}^{1}$-fibres of $\{ \xi_{p_{1}},\cdots,\xi_{p_{N}}\}$.

\item If $M$ is transverse $K$-stable, then $(M_{\infty},g_{\infty}^{T})$ is
isomorphic to $M$ endowed with a smooth Sasaki-Einstein metric.
\end{enumerate}
\end{theorem}

As a consequence, we have

\begin{corollary}
\label{C12} Let $M$ be a compact transverse Fano Sasakian $5$-manifold of Type
I foliation singularities. Then

\begin{enumerate}
\item There always exists a Sasaki-Ricci soliton metric of klt foliation singularities.

\item In addition, if $M$ is transverse $K$-stable, then $M$ endowed with a
smooth Sasaki-Einstein metric.
\end{enumerate}
\end{corollary}
\begin{remark}
\begin{enumerate}
\item Collins and Sz\'{e}kelyhidi (\cite{cz2}) called the $K$-stability of
Sasakian manifolds in the sense of its associated K\"{a}hler cone is
$K$-stable by following the definition of Donaldson.

\item Our Definition \ref{d61} for transverse $K$-stable works well through
the course for\ the studying of the Sasaki-Ricci flow. In fact, one can
introduce the generalized Sasaki-Futaki invariant defined by (\ref{59}) or
(\ref{59-1}) as in \cite{fow}, \cite{dt}, \cite{t3}, \cite{t5}, \cite{d} and
\cite{li}.

\item It was proved by Wang-Zhu (\cite{wz}) that all tori Fano K\"{a}hler
manifolds admit a K\"{a}hler-Ricci soliton metric and Futaki-Ono-Wang
(\cite{fow}) extended to tori Fano Sasakian manifolds. Also Shi-Zhu
(\cite{sz}) extended to tori Fano orbifolds as well. We also refer to
\cite{he} \ for the other existence theorem of Sasaki-Ricci solitons with
positive transversal bisectional curvature.

\item At the final remark, similar to Theorem \ref{T99}, we can confirm that
the Hamilton-Tian conjecture holds for the conic Sasaki-Ricci flow
(\ref{2026}) on a compact transverse log Fano Sasakian manifold of dimension
five with Type II foliation singularities as well (\cite{ccllw}) which will
appear elsewhere.
\end{enumerate}
\end{remark}
For the completeness, we refer to Appendix for the preliminary notions in this
paper which included the Sasakian structure, the Sasaki-Ricci flow, the log
pair foliation singularities and CR Bochner-Kodaira-Nakano identity over
Sasakian manifolds, etc.

In section $2$, we propose a Sasaki analogue of the conjecture picture of
analytic minimal model program (\cite{st}, \cite{sw}, \cite{jst}) for the
Sasaki-Ricci flow.

In section $3,$ we first conform the Hamilton-Tian conjecture for the
Sasaki-Ricci flow (\ref{2025}) in a compact transverse Fano quasi-regular
Sasakian $5$-manifold of Type I foliation singularities. The first central
issue is to show the $L^{4}$-bound of the transverse Ricci curvature under the
Sasaki-Ricci flow on transverse Fano Sasakian manifolds. Secondly, by applying
the argument as in \cite[Theorem 1.2]{tz1} to the limit leave orbifold space
$Z_{\infty}$ which is mainly depended on the Perelman's pseudolocality theorem
\cite{p1}, we have the smooth convergence of the Sasaki-Ricci flow on the
regular set $(M_{\infty})_{reg}$ which is a $\mathbb{S}^{1}$-principle bundle
over $\mathcal{R}$ and $\mathcal{R}$ is the regular set of $Z_{\infty}$.
Third, by the partial $C^{0}$-estimate, it is to show that $M_{\infty}$ is
isomorphic to a compact quasi-regular transverse Fano Sasaki-Ricci soliton
with at worst foliation log terminal singularities.

In section $4$, by applying the methods as in Guo-Phong-Song-Sturm
(\cite{gpss}), we have the compactness theorem of such Sasaki-Ricci solitons in transverse Fano
quasi-regular Sasakian $5$-manifolds. Then by the second structure theorem (Proposition \ref{P22}), we confirm the
Hamilton-Tian conjecture in a compact transverse Fano Sasakian $5$-manifold as
in Theorem \ref{T99}. The key issues is to obtain a bounded of a family of
Fano-type varieties due to Birkar. In fact, we will apply the algebraic and
analytic invariants : $\alpha$-invariant due to Tian-Zhang (\cite{tz1}) and
log canonical threshold due to Cheltsov-Shramov (\cite{cs}) and Demailly
(\cite{d1}), and a uniform positive lower bound of the log canonical threshold
due to Birkar (\cite{bir}).

In the last section, with its applications, we will show that the gradient
Sasaki-Ricci soliton orbifold metric on a compact Sasakian $5$-manifold is a
Sasaki-Einstein or conic Sasaki-Einstein metric if $M$ is transverse
$K$-stable, respectively. Note that the other $K$-stability condition proposed
by Collins and Sz\'{e}kelyhidi (\cite{cz2}) which is defined on its K\"{a}hler
cone. They showed that a polarized affine variety admits a Ricci-flat
K\"{a}hler cone metric if and only if it is $K$-stable. In particular, the
Sasakian manifold admits a Sasaki-Einstein metric if and only if its
associated K\"{a}hler cone is $K$-stable.

\section{Analytic Foliation Minimal Model Program}

In this section, we will propose a Sasaki analogue of analytic minimal model
program (\cite{st}, \cite{sw}, \cite{jst}) for the Sasaki-Ricci flow. It is
our goal to answer the topic $(1)$ of the conjecture picture as in subsection
$2.2.$ in a compact transverse Fano Sasakian $5$-manifold.

\subsection{Mori's Foliation Minimal Model Program}

One can \ view the Mori's minimal model program in birational geometry can be
viewed as the complex analogue of Thurston's geometrization conjecture which
was proved via Hamilton's Ricci flow with surgeries on $3$-dimensional
Riemannian manifolds by Perelman (\cite{p1}, \cite{p2}, \cite{p3}). Likewise,
there is a conjecture picture by Song-Tian (\cite{st}) that the
K\"{a}hler-Ricci flow should carry out an analytic minimal model program with
scaling on projective varieties. Recently, Song and Weinkove (\cite{sw})
established the above conjecture on a projective algebraic surface. In fact

(MMP) Finding an "optimal" representative of a projecyive variety $X$ within
its birational class :
\begin{enumerate}
\item Minimal model : The canonical bundle $K_{X}$\  \ is nef (nonpositively
curved,\ $K_{X}\cdot C\geq0,$ minimal, i.e. no $(-1)$-curve ); or
\item Mori fiber space : There exists a holomorphic map $\pi:X\rightarrow Y$
from $X$ to a lower dimensional variety $Y$ such that the generic fiber
$X_{y}=\pi^{-1}(y)$ is a manifold with $K_{X_{y}}<0$ (positively curved).
\end{enumerate}

More precisely, one expects to find a finite sequence of birational maps
$f_{1},...,f_{k\text{ }}$ and varieties $X_{1},...,X_{k}$ with
\[
X=X_{0}\overset{f_{1}}{\rightarrow}X_{1}\overset{f_{2}}{\rightarrow}%
X_{2}\overset{f_{3}}{\rightarrow}\cdot \cdot \cdot \overset{f_{k}}{\rightarrow
}X_{k}%
\]
so that either $X_{k}$ is minimal model or Mori fiber space. That is, to find
$f_{i}$ which "remove" curves $C$ with
\[
K_{X_{i}}\cdot C<0.
\]

In view of the previous discusses, it is natural to have a Sasaki analogue of
the Song-Tian conjecture picture (\cite{st}) that the Sasaki-Ricci flow will
carry out so-called an analytic foliation minimal model program on
quasi-regular Sasakian $5$-manifolds as well.

Let $(M,\xi_{0},\eta_{0},\Phi_{0},g_{0},\omega_{0})$ be a compact
quasi-regular Sasakian $5$-manifold and its leave space $Z$ of the
characteristic foliation be \textit{well-formed}.\textbf{\ }Then the orbifold
canonical divisor $K_{Z}^{orb}$ and canonical divisor $K_{Z}$ are the same and
thus
\[%
\begin{array}
[c]{c}%
K_{M}^{T}=\pi^{\ast}(\varphi^{\ast}(K_{Z})).
\end{array}
\]
In a such case, there is the Sasaki analogue of Mori's minimal model program
with respect to $K_{Z}$ in a such compact quasi-regular Sasakian $5$-manifold.
More precisely, one can ask the following so-called foliation minimal model
program :

Is there a foliation $(-1)$-curve ? One of Mori's program is to replace this
criterion with the one dictated by the transverse canonical divisor $K_{M}%
^{T}$ of $M$: Does the transverse canonical divisor have nonnegative
intersection
\[%
\begin{array}
[c]{c}%
K_{M}^{T}\cdot V\geq0
\end{array}
\]
with any invariant submanifold\ $V$ on $M$ which is a Sasakian $3$%
-submanifold. In other words,
\[
\mathrm{is}\  \ K_{M}^{T}\  \text{\ }\mathrm{nef\ }?
\]
\ If $K_{M}^{T}$ is not nef, there is an external transverse contraction map
which turns out not only to generalize the Sasaki analogue of Castelnuovo's
contractibility criterion but also to provide decisive information on the
global structures of the end results of the foliation minimal model program.
Then an end result of foliation MMP starting from $M$ is a transverse Mori
fiber space if and only if there exists a nonempty open set $U\subset M$ such
that for any $\mathbb{S}^{1}$-fibre in $U$, there is an irreducible invariant
$3$-submanifold $V$ passing through such a fibre $\mathbb{S}^{1}$
with$\ K_{M}^{T}\cdot V<0.$

\subsection{Conjecture Picture for Analytic FMMP}
The Sasaki analogue of the conjecture picture for the K\"{a}hler-Ricci flow,
we expect the following analytic foliation minimal model program on Sasakian
manifolds of dimension five for the Sasaki-Ricci flow.

Let $(M,\xi_{0},\eta_{0},\Phi_{0},g_{0},\omega_{0})$ be a compact
quasi-regular Sasakian $(2n+1)$-manifold of foliation quotient singularities
of Type I\textbf{. }We consider a solution $\omega=\omega(t)$ of the
Sasaki-Ricci flow

\[%
\begin{array}
[c]{c}%
\frac{\partial}{\partial t}\omega(t)=-\mathrm{Ric}^{T}(\omega(t)),\text{
}\omega(0)=\omega_{0}.
\end{array}
\]
As long as the solution exists, the cohomology class $[\omega(t)]_{B}$ evolves
by%
\[%
\begin{array}
[c]{c}%
\frac{\partial}{\partial t}\left[  \omega(t)\right]  _{B}=-c_{1}^{B}(M),\text{
}\left[  \omega(0)\right]  _{B}=\left[  \omega_{0}\right]  _{B},
\end{array}
\]
and solving this ordinary differential equation gives%
\[%
\begin{array}
[c]{c}%
\left[  \omega(t)\right]  _{B}=\left[  \omega_{0}\right]  _{B}-tc_{1}^{B}(M).
\end{array}
\]
We see that a necessary condition for the Sasaki-Ricci flow to exist for $t>0$
such that%
\[%
\begin{array}
[c]{c}%
\left[  \omega_{0}\right]  _{B}-tc_{1}^{B}(M)>0.
\end{array}
\]
This necessary condition is in fact sufficient. In fact, we define
\[%
\begin{array}
[c]{c}%
T_{0}:=\sup \{t>0\text{ }|\text{ }\left[  \omega_{0}\right]  _{B}-tc_{1}%
^{B}(M)>0\}.
\end{array}
\]
That is to say that
\begin{equation}%
\begin{array}
[c]{c}%
\left[  \omega_{0}\right]  _{B}-T_{0}c_{1}^{B}(M)\in \overline{C_{M}^{B}}%
\end{array}
\label{E}%
\end{equation}
which is a nef class.

\textbf{When }$K_{M}^{T}$\textbf{\ is not nef :} The flow must develop finite
singularities at $T_{0}<\infty:$ \ Then we start the initial Sasaki class
$[\omega_{0}]_{B}\in H_{B}^{2}(M,\mathbb{Q})$ with a pair $(M,H^{T})$ and then
the limit nef class
\[
L_{0}=H^{T}+T_{0}K_{M}^{T}%
\]
is a\textbf{\ }semi-ample $Q$-line bundle.

More precisely, it follows from the Kawamata base-point free theorem that
there exists a basic surjective transverse holomorphic map from $M$ into the
normal projective variety $Y$
\[
\Phi:M\rightarrow Y\subset(\mathbb{C}\mathrm{P}^{N},\omega_{FS})
\]
defined by the basic transverse holomorphic section $\{s_{0},s_{1}%
,\cdots,s_{N}\}$ of $H_{B}^{0}(M,(K_{M}^{T})^{m})$ which is $\mathbb{S}^{1}%
$-equivariant with respect to the weighted $\mathbb{C}^{\ast}$-action in
$\mathbb{C}^{N+1}$ with $N=\dim H^{0}(M,(K_{M}^{T})^{m})-1$ for a large
positive integer $m.$ In the quasi-regular case, we have the following diagram
: $\Phi=\pi_{N}\circ \overline{\Phi}$
\[%
\begin{array}
[c]{ccc}%
M & \overset{\overline{\Phi}}{\longrightarrow} & N\\
& \searrow^{\Phi} & \downarrow \pi_{N}\\
&  & Y.
\end{array}
\]
Now we have the following conjecture picture for the analytic foliation
minimal model program:

Let $(M,\xi_{0},\eta_{0},\Phi_{0},g_{0},\omega_{0})$ be a compact
quasi-regular Sasakian $5$-manifold of foliation quotient
\textbf{singularities of Type I. }
\begin{enumerate}
\item When $Y$ is an $0$-dimensional ($\dim N=1)$ : In this case, $M$ is
transverse Fano with $c_{1}^{B}(M)>0.$ Then by Perelman, the type $I$ blow-up
of solutions $g(t)$ of the Sasaki-Ricci flow has uniformly bounded diameter
and scalar curvature. That is, there exists $C>0$ such that for all $(x,t)\in
M\times \lbrack0,T_{0})$, we have%
\[
|R^{T}(x,t)|\leq \frac{C}{(T-t)}%
\]
and \
\[
\mathrm{Diam}(X,g(t))\leq C\sqrt{(T-t)}.
\]
This leads to the convergence of the Sasaki-Ricci flow to a unique singular
Sasaki-Ricci soliton on a transverse $Q$-Fano Sasakian $5$-manifold.
\item When $0<\dim Y<2$ ($\dim M>\dim N$; \textbf{Collapsed or fiber
contraction) :} $M$ is a transverse Mori fibration, i.e. the general fiber of
$\Phi$ is transverse Fano such that $\Phi=\pi_{N}\circ \overline{\Phi}:$
\[%
\begin{array}
[c]{ccc}%
M & \overset{\overline{\Phi}}{\longrightarrow} & N\\
& \searrow^{\Phi} & \downarrow \pi_{N}\\
&  & Y.
\end{array}
\]
The Sasaki-Ricci flow is conjectured to collapse transversely onto $Y$ and to
extend through $t=T$ on $N$ in the sense of the Gromov Hausdorff topology. For
any $q\in Y$ , the type $I$ blow-up based at $q$ should converge to a complete
non-flat Sasaki-Ricci soliton.
\item When $\dim Y=2$ ($\dim M=\dim N$; \textbf{Non-collapsed or divisorial
contraction) :} $K_{M}^{T}$ is big. It is conjectured that the flow will
extend through $t=T$ geometrically associated to a transverse birational
transform such as a divisorial contraction. Furthermore,
\begin{enumerate}
\item Divisorial contraction : $N$ differs from $M$ by a invariant submanifold
of codimension $2$, then we return to the previous step.
\item It follows up from the previous steps untill $T_{k}=\infty$ and then
$K_{M}^{T}$ is nef. $M_{k}$ has at worst foliation cyclic quotient
singularities (orbifold singularities) and has no foliation $K_{M}^{T}%
$-negative curves.
\end{enumerate}
\end{enumerate}

\begin{remark}
\begin{enumerate}
\item (\cite{cclw}) For the non-collapsing solutions in case $(3a)$, we
established the Gromov Hausdorff continuation through finite time
singularities on a compact quasi-regular Sasakian $5$-manifold of\ foliation
cyclic quotient singularities. More precisely, we have the results on the
analytic foliation minimal model program with scaling in a compact quasi-regular
Sasakian 5-manifold.

\item (\cite{chlw2}) For a smooth transverse \textbf{minimal model of general
type} as in $(3b)$, we proved that $(M,\omega(t))$ converges in the
Cheeger-Gromov sense to the unique Sasaki $\eta$-Einstein orbifold metric
$\omega_{\infty}$
\[%
\begin{array}
[c]{c}%
\mathrm{Ric}_{\omega_{\infty}}^{T}=-\omega_{\infty}%
\end{array}
\]
on the transverse canonical model $M_{\mathrm{can}\text{ }}$ with finite
orbifold foliation singularities at a singular fibre $S^{1}$ on $M$. In
particular, the floating foliation $(-2)$-curves in $M$ will be contracted to
orbifold points by the Sasaki-Ricci flow as $t\rightarrow \infty.$
\end{enumerate}
\end{remark}

Let $(M,\xi_{0},\eta_{0},\Phi_{0},g_{0},\omega_{0})$ be a compact
quasi-regular Sasakian $5$-manifold of foliation quotient singularities of
Type II and\textbf{\ }$(M,\Delta^{T})$ be a log pair with klt foliation
singularities. The corresponding foliation singularities in $(M,\xi_{0}%
,\eta_{0},\Phi_{0},g_{0},\omega_{0})$\ is the Hopf $\mathbb{S}^{1}$-orbibundle
over a Riemann surface $\Sigma_{h}$. The orbifold canonical divisor
$K_{Z}^{orb}$ and canonical divisor $K_{Z}$ are related by
\[%
\begin{array}
[c]{c}%
K_{Z}^{orb}=\varphi^{\ast}(K_{Z}+[\Delta])
\end{array}
\]
and then
\[%
\begin{array}
[c]{c}%
K_{M}^{T}=\pi^{\ast}(K_{Z}^{orb}).
\end{array}
\]
There is a Sasaki analogue of analytic log minimal model program via the conic
Sasaki-Ricci flow with the cone angle $2\pi \beta$ $(0<\beta<1)$ along the
basic divisor $\Delta^{T}\thicksim-K_{M}^{T}$ \ on $(M,\Delta^{T}%
)\times \lbrack0,T):$%
\begin{equation}
\left \{
\begin{array}
[c]{lcl}%
\frac{\partial}{\partial t}\omega(t) & = & -Ric^{T}(\omega(t))+(1-\beta
)\Delta^{T},\\
\omega(0) & = & \omega_{\ast}.
\end{array}
\right.  \label{2026}%
\end{equation}
Where $S^{T}$ is the basic defined section of $\Delta^{T}$ in transverse
regular log pair $(M,\Delta^{T})$ $\ $and $\omega_{\ast}=\omega_{0}%
+\delta \sqrt{-1}\partial_{B}\overline{\partial_{B}}||S^{T}||^{2\beta}$. We
refer to \cite{ccllw} for some details.

There is a Sasaki analogue of analytic log foliation minimal model program
with respect to $K_{M}^{T}+[D^{T}]$ via the conic Sasaki-Ricci flow which is
the odd dimensional counterpart of the conic K\"{a}hler-Ricci flow. More
precisely, one expects to find a finite sequence of transverse birational maps
$\Phi_{1},...,\Phi_{k\text{ }}$ and $M_{1},...,M_{k}$ with
\[
(M,D^{T})=(M_{0},D_{0}^{T})\overset{\Phi_{1}}{\rightarrow}(M_{1},D_{1}%
^{T})\overset{\Phi_{2}}{\rightarrow}(M_{2},D_{2}^{T})\overset{\Phi_{3}%
}{\rightarrow}\cdot \cdot \cdot \overset{\Phi_{k}}{\rightarrow}(M_{k},D_{k}^{T})
\]
so that either it is a Mori fiber space, or the canonical bundle $K_{M}%
^{T}+[D^{T}]$\  \ is nef. That is, to find $\Phi_{i}$ which "remove" foliation
curves $V$ which is an invariant $3$-dimensional submanifold of $M_{i}$ with
\[
(K_{M_{i}}^{T}+[D_{i}^{T}])\cdot V<0.
\]
Thus we expect the following conjecture picture for the analytic log foliation
minimal model program:

Let $(M,\xi_{0},\eta_{0},\Phi_{0},g_{0},\omega_{0})$ be a compact
quasi-regular Sasakian $5$-manifold of foliation quotient
\textbf{singularities of Type II }and\textbf{\ }$(M,D^{T})$ be a log pair with
klt foliation singularities. Then there exists a unique maximal conic
Sasaki-Ricci flow $\omega(t)$ on $(M_{0},D_{0}^{T}),(M_{1},D_{1}%
^{T}),..,(M_{k},D_{k}^{T}),$ starting at $(M_{0},D_{0}^{T},\omega_{0})$ with
floating foliation canonical surgical contractions of a finite number of
disjoint floating foliation $(-1)$-curves or foliation external ray
contractions of foliation $(K_{M}^{T}+[D^{T}])$-negative curves $\Phi
_{i}:(M_{i-1},D_{i-1}^{T})\rightarrow(M_{i},D_{i}^{T})$. In addition,
\begin{enumerate}
\item Either $T_{k}<\infty$ and the flow $\omega(t)$ collapses in the sense
that
\[%
\begin{array}
[c]{c}%
\lbrack(K_{M_{k}}^{T}+[D_{k}^{T}]]^{2}\rightarrow0
\end{array}
\]
as$\ t\rightarrow T_{k}^{-}.$
\begin{enumerate}
\item there exists a contraction%
\[%
\begin{array}
[c]{c}%
\Psi:M_{k}\rightarrow pt
\end{array}
\]
such that $-(K_{M_{k}}^{T}+[D_{k}^{T}])>0$ and thus $(M_{k},D_{k}^{T})$ is
transverse minimal Fano and the foliation space is a minimal log pair del
Pezzo surface with at worst klt foliation singularities, or
\item there exists a fibration
\[%
\begin{array}
[c]{c}%
\Psi:M_{k}\rightarrow \Sigma_{l},
\end{array}
\]
then $(M_{k},D_{k}^{T})$ is an $\mathbb{S}^{1}$-orbibundle of a log rule
surface over a\ Riemann surface $\Sigma_{l}$\ with finite marked points
$\{(p_{1},...,p_{l})\}$. i.e. The foliation space is a log rule surface over
a\ Riemann surface $\Sigma_{l}$\ with finite marked points. Or
\end{enumerate}
\item When $\dim M=\dim M_{i}$ (\textbf{Non-collapsed or divisorial
contraction) :} $K_{M_{k}}^{T}+[D_{k}^{T}]$ is big. It is conjectured that the
flow will extend through $t=T$ geometrically associated to a transverse
birational transform such as a divisorial contraction. Furthermore,
\begin{enumerate}
\item Divisorial contraction : $M_{i}$ differs from $M$ by a invariant
submanifold of codimension $2$, then we return to the previous step.
\item It follows up from the previous steps untill $T_{k}=\infty:$  $K_{M_{k}%
}^{T}+[D_{k}^{T}]$ is nef and $\Psi_{i}=\Psi_{k}$, and $(M_{k},D_{k}^{T})$ has
at worst klt foliation singularities and has no foliation $(K_{M}^{T}%
+[D^{T}])$-negative curves.
\end{enumerate}
\end{enumerate}
\section{Existence of Sasaki-Ricci solitons}
In the section, we consider the Sasakian manifold $(M,\eta,\xi,\Phi,g,\omega)$
with $2\pi c_{1}^{B}(M,\mathcal{F}_{\xi})=[\omega]$ and $\omega=\frac{1}%
{2}d\eta$ up to $D$-homothetic with respect to the Reeb vector field $\xi$. We
will focus on the regularity of the limit space and the partial $C^{0}%
$-estimate for the Sasaki-Ricci flow (\ref{2025}).

\subsection{Regularity of Sasaki-Ricci Solitons}
In our previous paper (\cite{chlw1}), we obtain the crucial estimate on the
$L^{4}$-bound of the transverse Ricci curvature under the Sasaki-Ricci flow
(\ref{2025}) in a compact transverse Fano Sasakian manifold $(M,\xi,\eta
_{0},\Phi_{0},g_{0},\omega_{0})$ of dimension five with foliation
singularities of type I.
\begin{theorem}
\label{T51} (\cite{chlw1}) Let $(M,\xi,\eta_{0},\Phi_{0},g_{0},\omega_{0})$ be
a compact transverse Fano Sasakian $(2n+1)$-manifold\textbf{.} Then, under the
Sasaki-Ricci flow (\ref{2025}), there exists a positive constant $C(g_{0})$
such that
\begin{equation}%
\begin{array}
[c]{c}%
\int_{M}|Ric_{\omega(t)}^{T}|^{4}\omega(t)^{n}\wedge \eta_{0}\leq C,
\end{array}
\label{40}%
\end{equation}
for all $t\in \lbrack0,\infty).$ That is it suffices to show that%
\[%
\begin{array}
[c]{c}%
\int_{M}|\nabla^{T}\overline{\nabla}^{T}u(t)|^{4}\omega(t)^{n}\wedge \eta
_{0}\leq C,
\end{array}
\]
for all $t\geq0$ and for some constant $C$ independent of $t$. Here $u(t)$ is
the evolving transverse Ricci potential.
\end{theorem}
\begin{remark}
\begin{enumerate}
\item Note that (\ref{40}) holds without the condition of
quasi-regular. It is the crucial step for our compactness theorem when the
sequence of Reeb vector fields
$\xi_{i}$ keeping in the same Type I or Type II of foliation singularities for
$n=2$.

\item Similar to Theorem \ref{T51} as in the paper of Chang-Chang-Li-Lin-Wu
(\cite{ccllw}), we have the $L^{4}$-boundedness of the transverse Ricci
curvature along the conic Sasaki-Ricci flow.
\end{enumerate}
\end{remark}
Based on Perelman's non-collapsing theorem for a transverse ball along the
unnormalized Sasaki-Ricci flow, it follows that

\begin{lemma}
\label{L61} (\cite[Proposition 7.2]{co1}, \cite[Lemma 6.2]{he}) Let
$(M^{2n+1},\xi,g_{0})$ be a compact Sasakian manifold and let $g^{T}(t)$ be
the solution of the unnormalized Sasaki-Ricci flow with the initial transverse
metric $g_{0}^{T}$. Then there exists a positive constant $C$ such that for
every $x\in M$, if $\ |R^{T}|\leq r^{-2}$ on $B_{\xi,g(t)}(x,r)$ for
$r\in(0,r_{0}]$, where $r_{0}$ is a fixed sufficiently small positive number,
then%
\[%
\begin{array}
[c]{c}%
\mathrm{Vol}(B_{\xi,g^{T}(t)}(x,r))\geq Cr^{2n}.
\end{array}
\]
Moreover, the transverse scalar curvature $R^{T}$ and transverse diameters
$d_{g^{T}(t)}^{T}$ are uniformly bounded under the Sasaki-Ricci flow. As a
consequence, there is a uniform constant $C$ such that
\[
\mathrm{diam}(M,g(t))\leq C.
\]

\end{lemma}
Now we are ready to study the structure of the limit space (\cite{chlw1}):

\begin{theorem}
\label{T61} Let $(M_{i},\eta_{i},\xi,\Phi_{i},g_{i})$ be a sequence of
quasi-regular Sasakian $(2n+1)$-manifolds with Sasaki metrics $g_{i}=g_{i}%
^{T}+\eta_{i}\otimes \eta_{i}$ such that for basic potentials $\varphi_{i}$
\[%
\begin{array}
[c]{c}%
\eta_{i}=\eta+d_{B}^{C}\varphi_{i}%
\end{array}
\]
and
\[%
\begin{array}
[c]{c}%
d\eta_{i}=d\eta+\sqrt{-1}\partial_{B}\overline{\partial}_{B}\varphi_{i}.
\end{array}
\]
We denote that $(Z_{i},h_{i},\Phi_{i},\omega_{h_{i}})$ are a sequence of
well-formed normal projective varieties which are the corresponding foliation
leave spaces with respect to $(M_{i},\eta_{i},\xi,\Phi_{i},g_{i})$ such that
\[%
\begin{array}
[c]{c}%
\omega_{g_{i}^{T}}=\frac{1}{2}d\eta_{i}=\pi^{\ast}(\omega_{h_{i}});\text{
}\Phi_{i}=\pi^{\ast}(J_{i})
\end{array}
\]
Suppose that $(M_{i},\eta_{i},\xi,\Phi_{i},g_{i})$ is a compact smooth
transverse Fano Sasakian $(2n+1)$-manifolds satisfying
\[%
\begin{array}
[c]{c}%
\int_{M}|Ric_{g_{i}^{T}}^{T}|^{p}\omega_{i}^{n}\wedge \eta \leq \Lambda,
\end{array}
\]
and
\[%
\begin{array}
[c]{c}%
\mathrm{Vol}(B_{\xi,g(t)}(x,r))\geq vr^{2n}%
\end{array}
\]
for some $p>n,$ $\Lambda>0,$ $\upsilon>0$. Then passing to a subsequence if
necessary, $(M_{i},\Phi_{i},g_{i},x_{i})$ converges in the Cheeger-Gromov
sense to limit length spaces $(M_{\infty},\Phi_{\infty},d_{\infty},x_{\infty
})$ and then $(Z_{i},J_{i},h_{i},\pi(x_{i}))$ converges to $(Z_{\infty
},J_{\infty},h_{\infty},\pi(x_{\infty}))$ such that

\begin{enumerate}
\item for any $r>0$ and $p_{i}\in M_{i}$ with $p_{i}\rightarrow p_{\infty}\in
M_{\infty},$%
\[%
\begin{array}
[c]{c}%
\mathrm{Vol}(B_{h_{i}}(\pi(p_{i}),r))\rightarrow \mathcal{H}^{2n}(B_{h_{\infty
}}(\pi(p_{\infty}),r))
\end{array}
\]
and
\[%
\begin{array}
[c]{c}%
\mathrm{Vol}(B_{\xi,g_{i}^{T}}(p_{i},r))\rightarrow \mathcal{H}^{2n}%
(B_{\xi,g_{\infty}^{T}}(p_{\infty},r)).
\end{array}
\]
Moreover,
\[%
\begin{array}
[c]{c}%
\mathrm{Vol}(B(p_{i},r))\rightarrow \mathcal{H}^{2n+1}(B(p_{\infty},r)),
\end{array}
\]
where $\mathcal{H}^{m}$ denotes the $m$-dimensional Hausdorff measure.

\item $M_{\infty}$ is a $\mathbb{S}^{1}$-orbibundle over $Z_{\infty}%
=M_{\infty}/\mathcal{F}_{\xi}.$ $Z_{\infty}=\mathcal{R}\cup \mathcal{S}$ such
that $\mathcal{S}$ is a closed singular set of complex codimension two and
$\mathcal{R}$ consists of points whose tangent cones are $\mathbb{R}^{2n}.$

\item the convergence on the regular part of $M_{\infty}$ which is a
$\mathbb{S}^{1}$-principle bundle over $\mathcal{R}$ in the $(C^{\alpha}\cap
L_{B}^{2,p})$-topology for any $0<\alpha<2-\frac{n}{p}$.
\end{enumerate}
\end{theorem}
By the convergence theorem in Theorem \ref{T61}, we have the regularity of the
limit space: We define a family of Sasaki-Ricci flow $g_{i}^{T}(t)$ by%
\[
(M,g_{i}^{T}(t))=(M,g^{T}(t_{i}+t))
\]
for $t\geq-1$ and $t_{i}\rightarrow \infty$ and for $g_{i}^{T}(t)=\pi^{\ast
}(h_{i}(t))$
\[
(Z,h_{i}(t))=(M,h_{i}(t_{i}+t)).
\]
Note that the flow (\ref{2025}) can be expressed locally as a parabolic
Monge-Amp\`{e}re equation on a basic K\"{a}hler potential $\varphi$ as in
(\ref{C}):%
\begin{equation}%
\begin{array}
[c]{c}%
\frac{d}{dt}\varphi=\log \det(g_{\alpha \overline{\beta}}^{T}+\varphi
_{\alpha \overline{\beta}})-\log \det(g_{\alpha \overline{\beta}}^{T}%
)+\varphi-u(0).
\end{array}
\label{2022-1}%
\end{equation}
Here $u(0)$ is the transverse Ricci potential of $\eta_{0}$, defined by%
\[
R_{kl}^{T}-g_{kl}^{T}=\partial_{k}\overline{\partial}_{l}u(0)
\]
which we normalize so that $\frac{1}{V}\int_{M}e^{-u(0)}\omega_{0}^{n}%
\wedge \eta_{0}=1.$ Let $u(t)$ be the evolving transverse Ricci potential.

It follows from \cite{co1} and \cite{chlw1} that
\begin{lemma}
\label{L31} Let $(M^{2n+1},\xi,g_{0})$ be a compact Sasakian manifold and let
$g^{T}(t)$ be the solution of the Sasaki-Ricci flow (\ref{2025}) with the
initial transverse metric $g_{0}^{T}$. Then there exists $C$ depending only on
the initial metric such that%
\[
||u(t)||_{C^{0}}+||\nabla^{T}u(t)||_{C^{0}}+||\Delta_{B}u(t)||_{C^{0}}\leq C
\]
for all $t\geq0.$
\end{lemma}
\begin{lemma}
(\cite{chlw1})\label{L41} Under the Sasaki-Ricci flow
\begin{equation}%
\begin{array}
[c]{c}%
\int_{M}|\nabla^{T}\nabla^{T}u|^{2}\omega(t)^{n}\wedge \eta_{0}\rightarrow0
\end{array}
\label{a2}%
\end{equation}
and
\begin{equation}%
\begin{array}
[c]{c}%
\int_{M}(\Delta_{B}u-|\nabla^{T}u|^{2}+u-a)^{2}\omega(t)^{n}\wedge \eta
_{0}\rightarrow0
\end{array}
\label{b2}%
\end{equation}
as $t\rightarrow \infty.$ Here $a(t)=\frac{1}{V}\int_{M}ue^{-u}\omega
(t)^{n}\wedge \eta_{0}$ with $\lim_{t\rightarrow \infty}a(t)=a_{\infty}\leq0.$
\end{lemma}
\begin{theorem}
\label{T63} Suppose (\ref{c2}) holds, then $g_{\infty}^{T}$ is smooth and
satisfies the Sasaki-Ricci soliton equation%
\begin{equation}
Ric^{T}(g_{\infty}^{T})+Hess^{T}(u_{\infty})=g_{\infty}^{T} \label{c1}%
\end{equation}
on $(M_{\infty})_{reg}$ which is a $\mathbb{S}^{1}$-bundle over $\mathcal{R}.$
Moreover, $\Phi_{\infty}$ is smooth and $g_{\infty}^{T}$ is K\"{a}hler with
respect to $\Phi_{\infty}=\pi^{\ast}(J_{\infty}).$
\end{theorem}
\begin{proof}
It follows from the convergence theorem \ref{T61} that, passing to a
subsequence if necessary, we have at $t=0,$%
\[
(M,g_{i}^{T}(0))\rightarrow(M_{\infty},g_{\infty}^{T},d_{\infty}^{T})
\]
such that
\[
(Z,h_{i}(0))\rightarrow(Z_{\infty},h_{\infty},d_{h_{\infty}})
\]
in the Cheeger-Gromov sense. Moreover,
\begin{equation}%
\begin{array}
[c]{c}%
(g_{i}^{T}(0),u_{i}(0))\overset{C^{\alpha}\cap L_{B}^{2,p}}{\rightarrow
}(g_{\infty}^{T},u_{\infty})
\end{array}
\label{c2}%
\end{equation}
on $(M_{\infty})_{reg}$ which is a $\mathbb{S}^{1}$-principle bundle over
$\mathcal{R}.$ The convergence of $u_{i}(0)$ follows from the elliptic
regularity (\cite{co1}, \cite{co2}) of
\[
\Delta_{B}u_{i}(0)=n-R^{T}(g_{i}^{T}(0))\in L_{B}^{p}.
\]

Since all $g_{\infty}^{T}$ and $u_{\infty}$ are basic, in the local harmonic
coordinate $\{t,x^{1},x^{2},\cdots,x^{2n}\}$, the Sasaki-Ricci soliton
equation (\ref{c1}) is equivalent to
\begin{equation}%
\begin{array}
[c]{c}%
(g^{T})^{\alpha \beta}\frac{\partial^{2}g_{\gamma \delta}^{T}}{\partial
x^{\alpha}\partial x^{\beta}}=\frac{\partial^{2}u_{\infty}}{\partial
x^{\gamma}\partial x^{\delta}}+Q(g^{T},\partial g^{T})_{\gamma \delta}%
+T(g^{-1},\partial g^{T},\partial u)_{\gamma \delta}-g_{\gamma \delta}^{T}.
\end{array}
\label{c3}%
\end{equation}

By (\ref{a2}), (\ref{c3}) holds in $L_{B}^{2}((M_{\infty})_{reg})$. But
$g_{\infty}^{T}$ and $u_{\infty}$ are $L_{B}^{2,p}$, then (\ref{c3}) holds in
$L_{B}^{p}((M_{\infty})_{reg})$ too. On the other hand, by (\ref{b2}) and
(\ref{c1}), we have that
\begin{equation}%
\begin{array}
[c]{c}%
g_{\alpha \beta}^{T}\frac{\partial^{2}u_{\infty}}{\partial x^{\alpha}\partial
x^{\beta}}=(g^{T})^{\alpha \beta}\frac{\partial u_{\infty}}{\partial x^{\alpha
}}\frac{\partial u_{\infty}}{\partial x^{\beta}}-2u_{\infty}+2a_{\infty}%
\end{array}
\label{c4}%
\end{equation}
in $L_{B}^{p}((M_{\infty})_{reg}).$

Then a bootstrap argument as in \cite{pe} to the elliptic systems (\ref{c3})
and (\ref{c4}) shows that $g_{\infty}^{T}$ and $u_{\infty}$ are smooth on
$(M_{\infty})_{reg}$. The elliptic regularity shows that $\Phi_{\infty}$ is
smooth since $\nabla_{g_{\infty}^{T}}^{T}\Phi_{\infty}=0.$
\end{proof}
By applying the argument as in \cite[Theorem 1.2]{tz1} to the normal orbifold
variety $Z_{\infty}$ which is mainly depended on the Perelman's pseudolocality
theorem \cite{p1}, we have the smooth convergence of the Sasaki-Ricci flow on
the regular set $(M_{\infty})_{reg}$ which is a $\mathbb{S}^{1}$-principle
bundle over $\mathcal{R}$ and $\mathcal{R}$ is the regular set of $Z_{\infty}$:

\begin{theorem}
\label{T64}The limit $(M_{\infty},d_{\infty})$ is smooth on the regular set
$(M_{\infty})_{reg}$ which is a $\mathbb{S}^{1}$-principle bundle over the
regular set $\mathcal{R}$ of $Z_{\infty}$ and $d_{\infty}^{T}$ is induced by a
smooth Sasaki-Ricci soliton $g_{\infty}^{T}$ and $g^{T}(t_{i})$ converge to
$g_{\infty}^{T}$ in the $C^{\infty}$-topology on $(M_{\infty})_{reg}.$
Moreover, the singular set $\mathcal{S}$ of $Z_{\infty}$ is the codimension
two orbifold singularities.
\end{theorem}
\begin{proof}
Note that since $g^{T}$ is a transverse K\"{a}hler metric and basic. It is
evolved by the K\"{a}hler-Ricci flow, then it follows from the standard
computations in K\"{a}hler-Ricci flow (\cite{co3}, \cite{mo}) that
\[%
\begin{array}
[c]{c}%
\frac{\partial}{\partial t}Rm^{T}=\Delta_{B}Rm^{T}+Rm^{T}\ast Rm^{T}+Rm^{T}.
\end{array}
\]
Now by Perelman's pseudolocality theorem (\cite{p1}, \cite{tz1}), there exists
$\varepsilon_{0},$ $\delta_{0},$ $r_{0}$ which depend on $p$ as in the Theorem
\ref{T61} such that for any $(x_{0},t_{0})$, if
\begin{equation}
\mathrm{Vol}(B_{\xi,g_{i}^{T}(t_{0})}(x_{0},r))\geq(1-\varepsilon
_{0})\mathrm{Vol}(B(0,r)) \label{e1}%
\end{equation}
for some $r\leq r_{0},$ where $\mathrm{Vol}(B(0,r)$ denotes the volume of
Euclidean ball of radius $r$ in $\mathbb{R}^{2n},$ then we have the following
curvature estimate%
\begin{equation}%
\begin{array}
[c]{c}%
||Rm^{T}||_{g_{i}^{T}}(x,t)\leq \frac{1}{t-t_{0}}%
\end{array}
\label{e2}%
\end{equation}
for all $x\in B_{\xi,g_{i}^{T}(t)}(x_{0},\varepsilon_{0}r)$ and $t_{0}<t\leq
t_{0}+\varepsilon_{0}^{2}r^{2}$ and the volume estimate
\begin{equation}
\mathrm{Vol}(B_{\xi,g_{i}^{T}(t)}(x_{0},\delta_{0}\sqrt{t-t_{0}}))\geq
(1-\eta)\mathrm{Vol}(B(0,\delta_{0}\sqrt{t-t_{0}})) \label{e3}%
\end{equation}
for $t_{0}<t\leq t_{0}+\varepsilon_{0}^{2}r^{2}$ and $\eta \leq \varepsilon_{0}$
is the constant such that the $C^{\alpha}$ harmonic radius at $x_{0}$ is
bounded below by $\delta_{0}\sqrt{t-t_{0}}.$

As in Theorem \ref{T61}, the metric $g_{i}^{T}(0)$ converges to $g_{\infty
}^{T}$ in the $(C^{\alpha}\cap L_{B}^{2,p})$-topology on $\mathcal{R}$. Now it
is our goal to show that the metric $g_{i}^{T}(0)$ converges smoothly to
$g_{\infty}^{T}$. For $0<r\leq r_{0}$ and $t\geq-1,$ define
\[
\Omega_{r,i,t}:=\{x\in M\text{ }|\text{ (\ref{e1}) holds on }B_{\xi,g_{i}%
^{T}(t)}(x,t)\}.
\]
Then (\ref{e3}) implies that
\[
\Omega_{r,i,t}\subset \Omega_{\delta_{0}\sqrt{s},i,t+s}%
\]
for $0<s\leq \varepsilon_{0}^{2}r^{2}.$

Let $r_{j}$ to be a decreasing sequence of radii such that $r_{j}\rightarrow0$
and $t_{j}=-\varepsilon_{0}r_{j}$. Then by applying \cite[(3.42)]{tz1}, we may
assume that
\[
\Omega_{r_{j},i,t_{j}}\subset \Omega_{r_{j+1},i,t_{j+1}}.
\]
Then by (\ref{e2})%
\begin{equation}%
\begin{array}
[c]{c}%
||Rm^{T}||_{g_{i}^{T}(t)}(x,t)\leq \frac{1}{t-t_{j}}%
\end{array}
\label{e4}%
\end{equation}
for all $(x,t)$ with
\[%
\begin{array}
[c]{c}%
d_{g_{i}^{T}(t)}^{T}(x,\Omega_{r_{j},i,t_{j}})\leq \varepsilon_{0}r_{j},\text{
}t_{j}<t\leq0.
\end{array}
\]
By Shi's derivative estimate \cite{shi} to the curvature, there exist a
sequence of constants $C_{k,j,i}$ such that%
\begin{equation}
||(\nabla^{T})^{k}Rm^{T}||_{g_{i}^{T}(0)}(x,t)\leq C_{k,j,i} \label{e5}%
\end{equation}
on $\Omega_{r_{j},i,t_{j}}.$ Then Passing to a subsequence if necessary, one
can find a subsequence $\{i_{j}\}$ of $\{j\}$ such that
\[%
\begin{array}
[c]{c}%
(\Omega_{r_{j},i_{j},t_{j}},g_{i_{j}}^{T}(t_{j}))\overset{C^{\alpha}%
}{\rightarrow}(\overline{\Omega},g_{\overline{\Omega}}^{T})
\end{array}
\]
and
\[%
\begin{array}
[c]{c}%
(\Omega_{r_{j},i_{j},t_{j}},g_{i_{j}}^{T}(0))\overset{C^{\infty}}{\rightarrow
}(\Omega,g_{\Omega}^{T}),
\end{array}
\]
where $(\overline{\Omega},g_{\overline{\Omega}}^{T})$ and $(\Omega,g_{\Omega
}^{T})$ are smooth Riemannian manifolds.
\[
(\Omega,g_{\Omega}^{T})\text{ is isometric to }((M_{\infty})_{reg},d_{\infty
}^{T}).
\]
On the other hand, as in Theorem \ref{T61}, we may also have
\[%
\begin{array}
[c]{c}%
(M,g_{i_{j}}^{T}(t_{j}))\overset{d_{G,H}^{T}.}{\rightarrow}(\overline
{M}_{\infty},\overline{d}_{\infty}^{T})
\end{array}
\]
with $\overline{Z}_{\infty}=\overline{\mathcal{R}}\cup \overline{\mathcal{S}}.$
Then%
\begin{equation}
(M_{\infty})_{reg}\text{ is the }\mathbb{S}^{1}\text{-principle bundle over
}\mathcal{R} \label{f1}%
\end{equation}
and
\begin{equation}
(\overline{M}_{\infty})_{reg}\text{ is the }\mathbb{S}^{1}\text{-principle
bundle over }\overline{\mathcal{R}}. \label{f2}%
\end{equation}
Moreover, by the continuity of volume under the Cheeger-Gromov convergence
(\cite[Claim 3.7]{tz1}), we have
\begin{equation}
(\overline{\Omega},g_{\overline{\Omega}}^{T})\text{ is isometric to
}((\overline{M}_{\infty})_{reg},\overline{d}_{\infty}^{T}) \label{f3}%
\end{equation}
and then follows from (\ref{e4}) as in \cite[Claim 3.8]{tz1} that
\begin{equation}
(\overline{\Omega},g_{\overline{\Omega}}^{T})\text{ is also isometric to
}(\Omega,g_{\Omega}^{T}). \label{f4}%
\end{equation}
Finally (\ref{f1}), (\ref{f2}), (\ref{f3}) and (\ref{f4}) imply the metric
$g_{i}^{T}(0)$ converges smoothly to $g_{\infty}^{T}$ on $\mathcal{R}$.
\end{proof}
The goal of the rest of the proof of Proposition \ref{T66} is to show that
$M_{\infty}$ is isomorphic to a compact quasi-regular transverse Fano
Sasaki-Ricci soliton with at worst foliation log terminal singularities.

\subsection{The Partial $C^{0}$-Estimate}

Let $(M,\eta,\xi,\Phi,g)$ be a sequence of quasi-regular Sasakian
$(2n+1)$-manifold. For the solution $(M,\omega(t),g^{T}(t))$ of the
Sasaki-Ricci flow and the line bundle $((K_{M}^{T})^{-1},h(t)=\omega^{n}(t))$
with the basic Hermitian metric $h(t)$ with respective to $\widetilde{g}%
^{T}(t)=e^{-\frac{1}{n}u}g^{T}(t),$ we work on the evolution of the basic
transverse holomorphic line bundle $((K_{M}^{T})^{-m},h^{m}(t))$ for a large
integer $m$ such that $(K_{M}^{T})^{-m}$ is very ample. We consider the basic
embedding (Proposition \ref{PCR}) which is $\mathbb{S}^{1}$-equivariant with
respect to the weighted $\mathbb{C}^{\ast}$-action in $\mathbb{C}^{N_{m}+1}$
\[
\Psi:M\rightarrow \Psi(M)\subset(\mathbb{CP}^{N_{m}},\omega_{FS})
\]
defined by the orthonormal basic transverse holomorphic section $\{ \sigma
_{0},\sigma_{1},\cdots,\sigma_{N}\}$ in $H_{B}^{0}(M,(K_{M}^{T})^{-m})$ with
$N_{m}=\dim H_{B}^{0}(M,(K_{M}^{T})^{-m})-1$ with
\[%
\begin{array}
[c]{c}%
\int_{M}(\sigma_{i},\sigma_{j})_{h^{m}(t)}\omega^{n}(t)\wedge \eta_{0}%
=\delta_{ij}.
\end{array}
\]
Define
\begin{equation}%
\begin{array}
[c]{c}%
\mathcal{F}_{m}(x,t):=\sum_{\alpha=0}^{N_{m}}||\sigma_{\alpha}||_{h^{m}}%
^{2}(x).
\end{array}
\label{58}%
\end{equation}
Note that under these notations (Appendix A.3), the curvature form of the
Chern connection $D$ on $(K_{M}^{T})^{-m}$ is
\[
Ric(h(t))=m\omega(t)
\]
and $\Delta_{B}^{T}=(g^{T})^{i\overline{j}}(t)\nabla_{\frac{\partial}{\partial
z^{i}}}\nabla_{\frac{\partial}{\partial \overline{z}^{j}}}$ on $(K_{M}%
^{T})^{-m}.$

The following result is a Sasaki analogue of the partial $C^{0}$-estimate
which was obtained in the K\"{a}hler-Ricci flow (\cite{ds}, \cite[Theorem
5.1]{tz1}, \cite{pss}) and the proof there does carry over to our Sasaki
setting due to the first structure theorem again on quasi-regular Sasakian
manifolds. For completeness, we give a sketch for the proof here.

\begin{theorem}
\label{T65} Suppose (\ref{c2}) holds, we have
\begin{equation}%
\begin{array}
[c]{c}%
\inf_{t_{i}}\inf_{x\in M}\mathcal{F}_{m}(x,t_{i})>0
\end{array}
\label{d11}%
\end{equation}
for a sequence of $m\rightarrow \infty$.
\end{theorem}
\begin{proof}
The main idea is that our basic function theory on quasi-regular Sasaki
manifolds follows immediately from the standard function theory on the
orbifold quotient. In fact, since the leave space of the characteristic
foliation $(Z,h,\omega_{h})$ is a normal projective variety with codimension
two orbifold singularities. As before, for an orbifold Riemannian submersion
$\pi:(M,g,\omega)\rightarrow(Z,h,\omega_{h})$ with $\omega=\pi^{\ast}%
(\omega_{h}).$ We consider the projection
\begin{equation}
\Pi:(C(M),\overline{g},J,\overline{\omega})\rightarrow(Z,h,\omega_{h})
\label{2022-b}%
\end{equation}
such that $\Pi|_{(M,g,\omega)}=\pi,$ then we have the relation between the
volume form of the K\"{a}hler cone metric on the metric cone and the volume
form of the Sasaki metric on $M$
\begin{equation}
i_{\frac{\partial}{\partial r}}\overline{\omega}^{n+1}=(\Pi^{\ast}\omega
_{h})^{n}\wedge \eta. \label{2022-a}%
\end{equation}
Here $\overline{\omega}^{n+1}=r^{2n+1}(\Pi^{\ast}\omega_{h})^{n}\wedge
dr\wedge \eta.$ Then, by integration over $M$ via (\ref{2022-a}) and
(\ref{2022-b}) as the similar computation in \cite[(6.6)]{chlw1}, this will
give the desired estimates back to the space $(Z,h,\omega_{h})$ which is the
same as in K\"{a}hler case.
Here we described briefly two main ingredients for the reader's completeness.
These are the gradient estimate to plurianti-canonical basic sections and
H\"{o}rmander's $L^{2}$-estimate to $\overline{\partial}_{B}$-operator on
basic $(0,1)$-forms. In case of the Ricci curvature bounded, these estimates
are standard as in \cite{ds} and \cite{t1}. In our situation, due to the lack
of the Ricci curvature bounded, the arguments should be modified as follows :

\textbf{(i).} The uniform bound of the Sobolev constant $C_{S}(g_{0}^{T},n)$
for the basic function along the Sasaki-Ricci flow was obtained as in
\cite[Theorem 1.1]{co2} :
\begin{equation}%
\begin{array}
[c]{c}%
\left(  \int_{M}f^{\frac{2(2n+1)}{2n-1}}\omega(t)^{n}\wedge \eta_{0}\right)
^{\frac{2n-1}{2n+1}}\leq C_{S}(g_{0}^{T},n)\int_{M}(||\nabla^{T}f||^{2}%
+f^{2})\omega(t)^{n}\wedge \eta_{0}%
\end{array}
\label{b4}%
\end{equation}
for every $f\in W_{B}^{1,2}(M)$. Moreover, we have the the basic Bochner
formula%
\[%
\begin{array}
[c]{ccl}%
\Delta_{B}^{T}|\nabla^{T}\sigma|^{2} & = & |\nabla^{T}\nabla^{T}\sigma
|^{2}+|\overline{\nabla}^{T}\nabla^{T}\sigma|^{2}-((n+2)m)|\nabla^{T}%
\sigma|^{2}\\
&  & +<Ric^{T}\langle \nabla^{T}\sigma,\cdot),\nabla^{T}\sigma \rangle
-2|\nabla^{T}\sigma|^{2}%
\end{array}
\]
and then, for $Ric^{T}=Ric+2g^{T}$
\[%
\begin{array}
[c]{ccl}%
\Delta_{B}^{T}|\nabla^{T}\sigma|^{2} & = & |\nabla^{T}\nabla^{T}\sigma
|^{2}+|\overline{\nabla}^{T}\nabla^{T}\sigma|^{2}-((n+2)m+1)|\nabla^{T}%
\sigma|^{2}\\
&  & +\langle \partial \overline{\partial}u(\nabla^{T}\sigma,\cdot),\nabla
^{T}\sigma \rangle
\end{array}
\]
for any basic transverse holomorphic section $\sigma \in H_{B}^{0}(M,(K_{M}%
^{T})^{-m}).$ Here $u$ is the transverse Ricci potential as before. This
together with Lemma \ref{L31} and the Sobolev inequality to make it possible
by applying the standard Nash-Moser iteration arguments (\cite[(i) and (ii) of
Proposition 4.1.]{pss}, \cite[(5.8)]{tz1}) to get
\begin{equation}%
\begin{array}
[c]{c}%
||\sigma||_{C^{0}}+\sqrt{m}||\nabla^{T}\sigma||_{C^{0}}\leq C(m^{\frac{n}{2}%
})\int_{M}||\sigma||^{2}\omega(t)^{n}\wedge \eta_{0}%
\end{array}
\label{d1}%
\end{equation}
for $\sigma \in$ $H_{B}^{0}(M,(K_{M}^{T})^{-m}).$

\textbf{(ii).} The $L^{2}$-estimate to $\overline{\partial}_{B}$-operator for
basic sections on quasi-regular Sasakian manifolds follows from the standard
function theory on the orbifold quotient :

There exists a $m_{0}$ such that for any basic transverse holomorphic section
$\sigma \in H_{B}^{0}(M,K_{M}^{T-m})$ and $m\geq m_{0}$ with
\[
\overline{\partial}_{B}\sigma=0,
\]
one can find a solution $\vartheta$
\[
\overline{\partial}_{B}\vartheta=\sigma
\]
satisfying the property%
\begin{equation}%
\begin{array}
[c]{c}%
\int_{M}||\vartheta||^{2}\omega(t)^{n}\wedge \eta_{0}\leq \frac{4}{m}\int
_{M}||\sigma||^{2}\omega(t)^{n}\wedge \eta_{0}.
\end{array}
\label{d2}%
\end{equation}
In fact, it follows from the CR Bochner-Kodaira-Nakano identity (Lemma
\ref{L2026}), Phong-Song-Sturm (\cite[(iii) of Proposition 4.1.]{pss}) and
Tian-Zhang (\cite[(5.9)]{tz1}) that
\[%
\begin{array}
[c]{c}%
\Delta_{\overline{\partial}_{B}}=\overline{\partial}_{B}\overline{\partial
}_{B}^{\ast}+\overline{\partial}_{B}^{\ast}\overline{\partial}_{B}\geq \frac
{m}{4}%
\end{array}
\]
on $C_{B}^{0}(M,T^{0,1}M\otimes(K_{M}^{T})^{-m})$ for a larger $m$.

Finally, we will apply (i) and (ii) to complete the proof of our theorem. Note
that when $r$ small enough, the transverse geodesic ball $B_{\xi,g}(x,r)$ is a
trivial $\mathbb{S}^{1}$-bundle over the geodesic ball $B_{h}(\pi(x),r)$ in
$(Z,h,\omega_{h},\pi(x))$ described as in Theorem \ref{T61} outside the
singular set of codimension $4$. Then the $L^{2}$-estimate to $\overline
{\partial}$-operator for sections on the K\"{a}hler case can be applied for
basic sections on quasi-regular Sasakian manifolds. Indeed, all the integrands
are only involved with the transverse K\"{a}hler structure $\omega(t)$ and
basic sections. Hence, under the Sasaki-Ricci flow, when one applies the
Weitzenb\"{o}ch type formulae and integration by parts, the expressions
involved behave essentially the same as in the K\"{a}hler-Ricci flow.

\textbf{(iii). }We will follow the method as in \cite[Theorem 5.1]{tz1} which
is a generalization of \cite{ds} in K\"{a}hler-Ricci curvature bounds. In
fact, as before $Z_{\infty}=\mathcal{R}\cup \mathcal{S}$ and $(M_{\infty
})_{reg}$ is a $\mathbb{S}^{1}$-principle bundle over the regular part
$\mathcal{R}$ of $Z_{\infty}$. As $r_{j}\rightarrow0,$ we can have a tangent
cone $\mathcal{C}_{x}=\lim_{j\rightarrow \infty}(Z_{\infty,}r_{j}^{-2}%
\omega_{\infty},x)$ of $(M_{\infty,}\omega_{\infty})$ at $x$. Now we can work
on $\mathcal{C}_{\xi_{x}}=\lim_{j\rightarrow \infty}(M_{\infty,}r_{j}%
^{-2}\omega_{\infty},\xi_{x})$ at $\xi_{x}$ as well. For any $\varepsilon>0$,
we denote
\[
V_{\xi}(x,\varepsilon)=\{y\in \mathcal{C}_{\xi_{x}}:y\in B_{\xi}(0,\varepsilon
^{1})\backslash \overline{B_{\xi}(0,\varepsilon)}),d_{\xi}(y,((M_{\infty
})_{\text{\textrm{sing}}})_{x})>\varepsilon \}.
\]
Here $h_{x}$ is the Ricci-flat cone metric on $\mathcal{C}_{x}\backslash
\mathcal{S}_{x}$ and $B_{\xi}$ is the transverse ball with respect to $h_{x}$.
Its K\"{a}hler form $\omega_{x}=\sqrt{-1}\partial \overline{\partial}\rho
_{x},\rho_{x}$ is the distance function from the vertex of $\mathcal{C}_{x}$
and $\mathcal{S}_{x}$ is complex codimension at least $2.$ Then for a large
$j,$ there are diffeomorphisms
\[
\phi_{j}:V_{\xi}(x,\frac{\varepsilon}{4})\rightarrow(M_{\infty})_{\mathrm{reg}%
}%
\]
such that
\[
\lim||r_{j}^{-2}\phi_{j}^{\ast}\omega_{\infty}-h_{x}||_{C^{k}(V_{\xi}%
(x,\frac{\varepsilon}{2}))}=0
\]
with $d(x,\phi_{j}(V_{\xi}(x,\varepsilon))<10\varepsilon r_{j}$ and $\phi
_{j}(V_{\xi}(x,\varepsilon))\subset B_{_{\xi}}(x,(1+\varepsilon^{-1}%
)r)\subset(M_{\infty,}\omega_{\infty}).$ Furthermore, for any $\delta>0,$
there is a section $\tau$ of $(K_{M_{\infty}}^{T})^{-m}$ on $\phi_{j}(V_{\xi
}(x,\varepsilon))=\phi(V_{\xi}(x,\varepsilon)),$ for a larger $m=k_{j}%
=r_{j}^{-2}$ and isomorphism $\psi:\mathcal{C}_{\xi_{x}}\times
\mathbb{C\rightarrow}(K_{M_{\infty}}^{T})^{-m}$ over $V_{\xi}(x,\varepsilon)$
such that
\begin{equation}
||\overline{\tau}||_{\infty}^{2}=e^{-\rho_{x}^{2}}\text{ \  \  \textrm{and}
\  \ }||\overline{\partial}\tau||_{\infty}\leq C\delta \label{2026-1}%
\end{equation}
with $\tau=\psi(\mathcal{C}_{\xi_{x}}\times1)$ and $||\cdot||_{\infty}$ is the
norm induced by $e^{-\frac{u_{\infty}}{n}}h_{\infty}.$ Furthermore, one can
extend $\tau$ to

$\tau$ $=\overline{\tau}$ on $\phi(V_{\xi}(x,\delta_{0}))$ and vanishes
outside $\phi(V_{\xi}(x,\varepsilon))$ such that there exists a constant
$\nu(\varepsilon,\delta_{0})>0$
\[
\int_{M_{\infty}}||\overline{\partial}\overline{\tau}||_{\infty}\omega
_{\infty}^{n}\wedge \eta_{0}\leq \nu r^{2n-2}%
\]
for a small $\nu$ as long as $\varepsilon,\delta$ are sufficiently small.
Therefore, there are diffeomorphisms
\[
\widetilde{\phi}_{i}:(M_{\infty})_{\mathrm{reg}}\rightarrow M
\]
and smooth isomorphisms
\[
F_{i}:(K_{M_{\infty}}^{T})^{-m}\rightarrow(K_{M}^{T})^{-m}%
\]
with $F_{i}(\overline{\tau})=\overline{\tau}_{i}$ on $\widetilde{\phi}%
(\phi(V_{\xi}(x,\delta_{0}))$ such that
\[
\int_{M}||\overline{\partial}\overline{\tau_{i}}||_{i}\omega^{n}(t_{i}%
)\wedge \eta_{0}\leq2\nu r^{2n-2}.
\]
Then by the $L^{2}$-estimate (\ref{d2}), we have a section $v_{i}$ of
$(K_{M}^{T})^{-m}$ such that%
\[
\overline{\partial}_{B}v_{i}=\overline{\partial}_{B}\overline{\tau}_{i}%
\]
satisfying the property%
\[%
\begin{array}
[c]{c}%
\int_{M_{{}}}||v_{i}||_{i}^{2}\omega(t)^{n}\wedge \eta_{0}\leq \frac{1}{m}%
\int_{M}||\overline{\partial}\overline{\tau}_{i}||_{i}^{2}\omega(t)^{n}%
\wedge \eta_{0}\leq3\nu r^{2n}%
\end{array}
\]
for a large $i$. Let
\[
\sigma_{i}=\overline{\tau}_{i}-v_{i}.
\]
By the standard elliptic estimate, we have
\[
\sup_{\widetilde{\phi}(\phi(V_{\xi}(x,2\delta_{0})\cap B_{\xi}(o,1))}%
||v_{i}||^{2}\leq C(\delta_{0}r)^{-2n}\int_{M_{{}}}||v_{i}||_{i}^{2}%
\omega(t_{i})^{n}\wedge \eta_{0}\leq C\nu \delta_{0}{}^{-2n}.
\]
For any given $\delta_{0},$ if $\delta$ and $\varepsilon$ are small, then we
can choose smaller $\nu$ such that
\[
8C\nu \leq \delta_{0}{}^{2n}.
\]
It follows from (\ref{2026-1}) and the choice of $\overline{\tau}$ that
\[
|\sigma_{i}|_{i}\geq|F_{i}(\overline{\tau})|_{i}-|v_{i}|_{i}\geq \frac{1}{2}%
\]
on $\widetilde{\phi_{i}}(\phi(V_{\xi}(x,\delta_{0})\cap B_{\xi}(o,1))$ for a
large $i$. But by applying (\ref{d1}) to $\sigma_{i}$, for a larger $m=r^{-2}%
$, we have
\[
\sup_{M}|\nabla \sigma_{i}|_{i}\leq Cm^{\frac{n+1}{2}}(\int_{M_{{}}}%
||\sigma_{i}||_{i}^{2}\omega(t_{i})^{n}\wedge \eta_{0})^{\frac{1}{2}}\leq
Cr^{-1}.
\]
Since $d(x,\phi(\delta_{0}p_{0}))\leq10\delta_{0}r$ for some point $p_{0}%
\in \partial B_{\xi}(0,1),$ we have, for $i$ sufficiently large,
\[
|\sigma_{i}|_{i}(x_{i})\geq \frac{1}{4}-C\delta_{0}.
\]

Hence choose $\delta_{0}$ such that $C\delta_{0}<\frac{1}{8}$, it follows
that
\[
\mathcal{F}_{m}(x_{i},t_{i})>\frac{1}{8}%
\]
for a sequence of $m\rightarrow \infty.$ Then we are done.
\end{proof}
\textbf{The Proof of Proposition \ref{T66} :} As a consequence of the first
structure theorem for Sasakian manifolds and Theorem \ref{T65}, the
Gromov-Hausdorff limit $Z_{\infty}$ is a variety embedded in some
$\mathbb{CP}^{N}$ and the singular set $\mathcal{S}$ is a subvariety
(\cite{ds}, \cite[Theorem 1.6]{t2}). Then one can refine the regularity of
Theorem \ref{T64}, by the first structure theorem for Sasakian manifolds and
the partial $C^{0}$-estimate, that the Gromov-Hausdorff limit $Z_{\infty}$ is
a variety embedded in some $\mathbb{CP}^{N}$ and the singular set
$\mathcal{S}$ is a normal subvariety (\cite{ds}, \cite[Theorem 1.6]{t2},
\cite{tz1}, \cite{pss}). This will imply the Proposition \ref{T66}.

\section{Compactness Theorem}

In this section, we study the Sasaki analogue of compactness theorem of
K\"{a}hler-Ricci solitons. We first recall the second structure theorem on Sasakian manifolds.
\begin{proposition}
\label{P22}(\cite{ru}) Let $(M,g)$ be a compact Sasakian manifold of dimension
$2n+1$. Any Sasakian structure $(\xi,\eta,\Phi,g)$ on $M$ is either
quasi-regular or there is a sequence of quasi-regular Sasakian structures
$(\xi_{i},\eta_{i},\Phi_{i},g_{i})$ converging in the compact-open $C^{\infty
}$-topology to $(\xi,\eta,\Phi,g).$ In particular, if $M$ admits an irregular
Sasakian structure, it admits many locally free circle actions.
\end{proposition}

By the second structure theorem of compact Sasakian $5$-manifolds, it suffices
to focus on the compactness properties of a sequence of Sasaki-Ricci solitons
on a compact quasi-regular transverse Fano Sasakian $5$-manifold and its leave
space is normal projective Fano variety with Kawamata log terminal (klt)
singularities. More precisely, we will derive the Sasaki analogue of
compactness theorems of compact K\"{a}hler-Ricci solitons.

Let $(M,\xi_{i},\eta_{i},\Phi_{i},g_{i})$ be a sequence of quasi-regular
Sasakian manifolds. It follows from the second structure theorem
(\cite[Theorem 11.1.7]{bg}) that the irregular Reeb vector field $\xi$ lies in
the commutative Lie subalgebra $\emph{t}(M,\mathbb{F}_{\xi})\subset
\emph{a}(M,g)$ which has dimension $k>1$, there are $\rho_{i}\in
\emph{t}(M,\mathbb{F}_{\xi})$ and smooth positive function
\[%
\begin{array}
[c]{c}%
f_{i}:=\frac{1}{1+\eta(\rho_{i})}%
\end{array}
\]
with $\lim_{i\rightarrow \infty}f_{i}\nearrow1$ and $\lim_{i\rightarrow \infty
}\rho_{i}=0$ such that%
\[%
\begin{array}
[c]{c}%
\xi_{i}=\xi+\rho_{i}\text{ \  \textrm{and} \ }\eta_{i}=\frac{1}{1+\eta(\rho
_{i})}\eta.
\end{array}
\]
Moreover for $\ker \eta_{i}=\ker \eta=D$ and $g=g_{D}+\eta \otimes \eta=\frac
{1}{2}d\eta+\eta \otimes \eta$
\[
g_{i}=f_{i}[g_{D}+f_{i}\eta \otimes \eta]+df_{i}\circ \Phi \wedge \eta-f_{i}%
d\eta \circ(\Phi \rho_{i}\otimes \eta)\otimes \mathbb{I}.
\]
As $i$ larger enough, $g_{i}$ are well defined Riemannian metrics on $M$ which
converge to $g$ with respect to the compact-open $C^{\infty}$-topology and the
transverse Ricci tensor can be read as%
\[
Ric^{T}(g_{i})=\lambda_{i}Ric^{T}(g)+A(f_{i},\rho_{i},g)
\]
where $A(f_{i},\rho_{i},g)$ is a symmetric $2$-tensor depending on $f_{i}%
,\rho_{i},g$ with
\[%
\begin{array}
[c]{c}%
\lim_{i\rightarrow \infty}A(f_{i},\rho_{i},g)=0
\end{array}
\]
and $\lambda_{i}\in C^{\infty}(M)$ with
\[%
\begin{array}
[c]{c}%
\lim_{i\rightarrow \infty}\lambda_{i}=1.
\end{array}
\]

Next we recall a stability result of the Sasaki-Ricci flow as our starting point,

\begin{proposition}
\label{P2026} (\cite{he}) Let $(M,\xi,\eta,\Phi,g_{0})$ be a compact irregular
transverse Fano Sasakian manifold and $(M,\xi^{i},\eta_{i},\Phi_{i},g^{i})$ be
a sequence of quasi-regular Sasakian manifolds such that $(M,\xi^{i},\eta
_{i},\Phi_{i},g^{i})\rightarrow(M,\xi,\eta,\Phi,g_{0})$ in the compact-open
$C^{\infty}$-topology. Suppose that $g(t)$ and $g^{i}(t)$ be the solutions of
the Sasaki-Ricci flow (\ref{2025}) with the initial metric $g_{0}$ and
$g^{i},$ respectively. Then for any $\varepsilon$ and $T\in(0,\infty)$, there
is a constant $N(\varepsilon,k,T)$ such that for $i\geq N$ and $t\in
\lbrack0,T]$,%
\[%
\begin{array}
[c]{c}%
||g^{i}(t)-g(t)||_{C^{k}(M,g)}\leq \varepsilon.
\end{array}
\]

\end{proposition}

Now we are ready to study the Sasaki analogue of compactness theorem of
K\"{a}hler-Ricci solitons due to Tian-Zhang \cite{tz1}, Guo-Phong-Song-Sturm
\cite{gpss}, Cheltsov-Shramov \cite{cs}, Demailly \cite{d1}, Birkar
\cite{bir}, and then the second Sasakian structure theorem.
\begin{theorem}
\label{T55} (Compactness Theorem) Let $\{(M_{i},\xi_{i},\eta_{i},g_{i}%
,V_{i})\}$ be a sequence of Sasaki-Ricci soliton on compact quasi-regular
transverse Fano Sasakian $5$-manifold and its leave space is normal projective
Fano variety with codimension two Kawamata log terminal (klt) singularities as
in Proposition \ref{T66}. Then after possibly passing to subsequence,
$(M_{i},g_{i})$ converges in the Cheeger-Gromov topology to a compact metric
length space $(M_{\infty},d_{\infty})$ with the transverse distance function
$d_{\infty}^{T}$ satisfying the following

\begin{enumerate}
\item $(M_{i},\xi_{i},\eta_{i},g_{i}^{T},V_{i})$ converges smoothly to a
compact gradient Sasaki-Ricci soliton $(M_{\infty},\xi,\eta_{\infty}%
,g_{\infty}^{T},V_{\infty})$. Moreover, $g_{i}^{T}$ converges to a smooth
Sasaki-Ricci soliton metric in the Cheeger-Gromov topology on $M_{\infty
}\backslash \Sigma_{\infty}^{T}$ and $(M_{\infty},d_{\infty}^{T})$ coincides
with the metric completion of $(M_{\infty}\backslash \Sigma_{\infty}%
^{T},g_{\infty}^{T})$. Here $\Sigma_{\infty}^{T}$ is the singular set of
$\mathbb{S}^{1}$-fibres of $\{ \xi_{p_{1}},\cdots,\xi_{p_{N}}\}$.

\item $(M_{\infty},\xi,\eta_{\infty},g_{\infty}^{T},V_{\infty})$ is a
transverse $\mathbb{Q}$-Fano Sasakian manifold with foliation Kawamata log
terminal singularities at $\Sigma_{\infty}^{T}$.

\item The Sasaki-Ricci soliton metric $g_{\infty}^{T}$ extends to a transverse
K\"{a}hler current on $M_{\infty}$ with bounded local potential
and\ $V_{\infty}$ extends to a global Hamiltonian holomorphic vector field on
$M_{\infty}$.
\end{enumerate}
\end{theorem}
The key ingredient is to apply the algebraic and analytic Invariants: $\alpha
$-invariant due to Tian and log canonical threshold due to Cheltsov-Shramov
and Demailly, and a uniform positive lower bound of the log canonical
threshold due to Birkar.

\begin{definition}
Let $M$ be an $(2n+1)$-dimensional compact transverse Fano Sasakian manifold
with an ample basic transverse holomorphic line bundle $((K_{M}^{T})^{-1},h)$
with the curvature $\omega$ which is a transverse K\"{a}hler metric in
$c_{1}^{B}(M).$

\begin{enumerate}
\item We define transverse Tian's $\alpha$-invariant (\cite{zh1})
\[%
\begin{array}
[c]{ccl}%
\alpha(M,\mathcal{F}_{\xi}) & = & \sup \{ \alpha>0|\text{ }\exists \text{
}C(\alpha)>0,\text{ such that }\int_{M}e^{-\alpha(\varphi-\sup \varphi)}%
\omega^{n}\wedge \eta \leq C,\\
&  & \text{for all }\varphi \in C_{B}^{\infty}(M,\mathbb{R})\text{ with }%
\omega+\sqrt{-1}\partial_{B}\overline{\partial}_{B}\varphi>0\}.
\end{array}
\]

\item The transverse Demailly-Koll\'{a}r complex singularity exponent for the
quasi-regular Sasakian manifold $M$ (\cite{dk})
\[%
\begin{array}
[c]{ccl}%
c(\varphi) & = & \sup \{ \alpha>0|\text{ }\int_{M}e^{-\alpha \varphi}\omega
^{n}\wedge \eta<\infty,\text{ for the upper }\\
&  & \text{semicontinuous basic }L^{1}\text{-function }\varphi \text{ with
}\omega \text{-PSH}\\
&  & \text{in the sense of distribution with }\sup_{M}\varphi=0\}.
\end{array}
\]
\end{enumerate}
\end{definition}
\begin{theorem}
\label{T2025} Let $(M,\eta,\xi,\Phi,g,\omega)$ be an $(2n+1)$-dimensional
compact Sasakian manifold with $2\pi c_{1}^{B}(M,\mathcal{F}_{\xi})=[\omega]$.
If
\[%
\begin{array}
[c]{c}%
\alpha(M,\mathcal{F}_{\xi})>\frac{n}{n+1},
\end{array}
\]
then there exists a Sasakian-Einstein metric on $M$.\ That is, $(\ast)_{t}$
can be solved for all $t\in \lbrack0,\frac{n+1}{n}\alpha(M)).$
\end{theorem}
\begin{proof}
For a basic function $F$ with%
\[
\rho^{T}=\omega+\sqrt{-1}\partial_{B}\overline{\partial}_{B}F,
\]
the existence of Sasaki-Einstein metric with $c_{1}^{B}>0$ can be reduced to
find a new Sasaki-metric $\widetilde{\omega}=\omega+\sqrt{-1}\partial
_{B}\overline{\partial}_{B}\varphi$ such that
\[%
\begin{array}
[c]{c}%
\log \frac{\widetilde{\omega}^{n}\wedge \eta}{\omega^{n}\wedge \eta}%
=-\varphi+F\text{, }\int_{M}\varphi \omega^{n}\wedge \eta=0.
\end{array}
\]
By the continuity method, we consider the following a family of equations%
\begin{equation}%
\begin{array}
[c]{c}%
\log \frac{\widetilde{\omega}^{n}\wedge \eta}{\omega^{n}\wedge \eta}=-t\varphi+F
\end{array}
\tag{$(\ast)_t$}%
\end{equation}
for $t\in \lbrack0,1].$ By the Harnack estimate as in the paper of Zhang
\cite{zh1} and Demailly-Koll\'{a}r \cite{dk} which is due to Siu \cite{siu}
and Tian \cite{t3}, $(\ast)_{t}$ can be solved up to $t$ if
\[%
\begin{array}
[c]{c}%
\int_{M}e^{-\gamma t\varphi}\omega^{n}\wedge \eta<\infty
\end{array}
\]
for any choice of $\gamma \in(\frac{n}{n+1},1).$ Then our Theorem \ref{T2025}
follows easily.
\end{proof}
It follows easily from Demailly-Koll\'{a}r \cite{dk} that

\begin{corollary}
Let $(M,\eta,\xi,\Phi,g,\omega)$ be an $(2n+1)$-dimensional compact transverse
Fano quasi-regular Sasakian manifold with $2\pi c_{1}^{B}(M,\mathcal{F}_{\xi
})=[\omega]$. Then
\begin{equation}
\alpha(M,\mathcal{F}_{\xi}):=\inf \{c(\varphi)|\varphi \text{ is }%
\omega \text{-PSH}\} \text{.} \label{1}%
\end{equation}
\end{corollary}

\begin{definition}
(\cite{cs}, \cite{d1}) A log pair $(Z,\Delta)$ is consisting of a connected
compact normal projective variety $Z$ and an effective $\mathbb{Q}$-divisor
$\Delta$ such that $(K_{Z}+\Delta)$ is $\mathbb{Q}$-Cartier. Let $D$ be a
effective $\mathbb{Q}$-Cartier divisor on $Z.$ The critical value of $t$ such
that $(Z,\Delta+tD)$ is log canonical for a small $t$ but not klt for a larger
$t$ is called the log canonical threshold with respective to $D.$ For the
simplicity, let $(Z,\emptyset)$ be a normal projective variety with klt
singularities and $D$ be a effective $\mathbb{Q}$-Cartier divisor. The log
canonical threshold of $D$ is defined by%
\[
lct(Z,D):=\sup \{t\in \mathbb{Q}\mathbf{:}\text{ }(Z,tD)\text{ \textrm{is log
canonical}}\} \text{.}%
\]
\end{definition}
Now we consider a compact transverse Fano quasi-regular Sasakian manifold and
its leave space which is a normal projective orbifold $Z$ as above. Then the
submersion $\pi:(M,D^{T})\rightarrow(Z,D)$ is the corresponding $\mathbb{S}%
^{1}$-orbibundle over the normal projective orbifold $Z$ and $D^{T}$ is the
corresponding transverse effective $\mathbb{Q}$-divisor such that $\pi_{\ast
}D^{T}=D.$

\begin{definition}
(\ref{d63}) Let $M$ be a compact transverse Fano quasi-regular Sasakian
manifold and its leave space which is a normal projective orbifold $Z,$ the
corresponding log canonical threshold of $D^{T}$ is then defined by%
\[
lct(M,D^{T}):=\sup \{t\in \mathbb{Q}\mathbf{:}\text{ }(M,tD^{T})\text{
\textrm{is log canonical}}\} \text{.}%
\]
\end{definition}

Let $M$ be a compact transverse Fano quasi-regular Sasakian manifold. To
relate $\alpha(M)$ to algebraic geometry, let $D^{T}$ be any nonzero basic
divisor in the linear series $|m(K_{M}^{T})^{-1}|$. Therefore there is a
global nonzero basic transverse holomorphic section $S^{T}$ of $H^{0}%
(M,(K_{M}^{T})^{-m})$ with zero basic divisor equal to $D^{T}$. Since $S^{T}$
is unique up to scaling, we will rescale it with respect to $((K_{M}^{T}%
)^{-m},h^{m},m\omega)$ such that $\int_{M}||S^{T}||_{h^{m}}^{2}\omega
^{n}\wedge \eta=1.$

Note that the function $\frac{1}{m}\log||S^{T}||_{h^{m}}^{2}$ is $\omega$-PSH
due to the transverse Poincar\'{e}-Lelong formula%
\[%
\begin{array}
[c]{c}%
\omega+\frac{1}{m}\sqrt{-1}\partial_{B}\overline{\partial}_{B}\log
||S^{T}||_{h^{m}}^{2}=\frac{1}{m}[D^{T}],
\end{array}
\]
where $[D^{T}]$ is the current of integration along $D^{T}$.

\begin{theorem}
(\cite{cs}, \cite{d1}) Let $M$ be a compact transverse Fano quasi-regular
Sasakian manifold of foliation klt singularities. The $\alpha$-invariant and
log canonical threshold of $D^{T}$ is related by
\[%
\begin{array}
[c]{ccl}%
\alpha(M) & = & \inf_{m\geq1}\inf_{D^{T}\in|m(K_{M}^{T})^{-1}|}lct(M,\frac
{1}{m}D^{T})\\
& = & \sup \{t\in \mathbb{Q}\text{ }\mathbf{|}\text{ }(M,tD^{T})\text{
\textrm{is log canonical and }}D^{T}\sim-K_{M}^{T}\}.
\end{array}
\]
\end{theorem}
\begin{proof}
For a quasi-regular Sasakian manifold $M,$ its leave space is a normal
projective orbifold $Z$ with klt singularities. One can work on an orbifold
$Z$ instead of a non-singular variety. Then we can follow the same method due
to Demailly's method (\cite{cs}, \cite{d1}) which works as well on a normal
projective variety. More precisely, the method of $L^{2}$-estimates can be
extended to the case of orbifold singularities by using the orbifold metric on
the open set of regular points of $Z$ instead. Then we are done.
\end{proof}
On the other hand, Birkar obtains a uniform positive lower bound of the log
canonical threshold as following.

\begin{proposition}
(\cite[Corollary 1.4]{bir}) Let $\varepsilon$ be the positive real number.
Then normal projective varieties $Z$ of dimension $n$ such that the klt pair
$(Z,\Delta)$ is $\varepsilon$-lc for some boundary $B$ and $-(K_{Z}+\Delta)$
is nef and big, form a bounded of family.
\end{proposition}

In particular $Z$, as the leave space of a quasi-regular Sasakian manifold, is
an $n$-dimensional normal Fano projective orbifold with $\varepsilon$-lc
singularities. Hence

\begin{corollary}
\label{C2025-2}There exists a positive constant $\alpha_{0}(n)$ such that for
any quasi-regular transverse Fano Sasakian manifold $M$
\[
\alpha(M)\geq \alpha_{0}(n).
\]

\end{corollary}

\begin{remark}
Note that any Kawamata log terminal (klt) singularity is $\varepsilon$-lc for
some positive real number $\varepsilon$. Since the well-formed leave space
$(Z,\emptyset)$ only admits orbifold singularities, $Z$ is automatically klt.
For the leave space which is not well-formed, we assume that $(Z,\Delta)$ is a
klt log pair.
\end{remark}
As a consequence of Theorem \ref{T2025} and Corollary \ref{C2025-2} that

\begin{theorem}
There exists a positive constant $\alpha_{1}(n)$ such that for any
quasi-regular transverse Fano Sasakian manifold $M$, there exists a transverse
K\"{a}hler metric $\widehat{\omega}\in c_{1}^{B}(M)$ satisfying%
\begin{equation}
Ric^{T}(\widehat{\omega})\geq \alpha_{1}(n)\widehat{\omega}. \label{2}%
\end{equation}
\end{theorem}
We can assume that $\widehat{\omega}$ is invariant under the maximal group $G$
action of $\mathrm{Aut}(M)$ so that $\omega$ is a $G$-invariant transverse
K\"{a}hler metric in the equation $(\ast)_{t}$.

\begin{theorem}
There exists $A(n)>0$ such that
\begin{equation}
\mu(\widehat{g}^{T})\geq-A, \label{3}%
\end{equation}
for the transverse K\"{a}hler metric $\widehat{g}^{T}$ associated to the
$\widehat{\omega}$ as in (\ref{2}).
\end{theorem}

\begin{proof}
We recall the transverse $\mathcal{W}$-functional as in the paper of
Chang-Han-Lin-Wu (\cite{chlw1})
\[%
\begin{array}
[c]{c}%
\mathcal{W}^{T}(g^{T},f)=\frac{1}{V}\int_{M}(R^{T}+|\nabla^{T}f|^{2}%
+f-2n)e^{-f}d\mu,
\end{array}
\]
where $d\mu=\omega^{n}\wedge \eta$ and $V=(c_{1}^{B}(M))^{n}$, and the
transverse $\mu$-functional
\[%
\begin{array}
[c]{c}%
\mu(g^{T})=\inf_{f}\{ \mathcal{W}^{T}(g^{T},f):\frac{1}{V}\int_{M}e^{-f}%
d\mu=1\}.
\end{array}
\]

It follows from (\ref{2}) that we have Myers' and volume comparison as
\[
\mathrm{Vol}(M,\widehat{g}^{T})\leq C(n)
\]
and
\[
\mathrm{Diam}(M,\widehat{g}^{T})\leq C(n).
\]
Moreover, $\mathrm{Vol}(M,\widehat{g}^{T})\geq \overline{C}(n)$ due to
$\widehat{\omega}\in c_{1}^{B}(M).$ Then it follows from the previous
Corollary and Croke's theorem \cite{cr} that we have the uniformly bound of
the Sobolev constant $C_{S}$ of $(M,\widehat{g}^{T}).$ It is well-known that a
Sobolev inequality implies the lower bound of $\mu$-functional. Hence the
lower bound of the transverse $\mu$-functional as well.
\end{proof}
On the other hand, based on the paper of Chang-Han-Lin-Zhou (\cite{chlj}), we have

\begin{proposition}
\label{P77} Let $(M,\eta,\xi,\Phi,\omega^{T})$ be a compact Fano Sasakian
manifold which admits a Sasaki-Ricci soliton $(g_{SS}^{T},V)$. Then any
solution $g^{T}(x,t)$ of the normalized Sasaki-Ricci flow (\ref{2025}) will
converge to $g_{SS}^{T}(x)$ in the sense of Cheeger-Gromov $C^{\infty}%
$-topology with $\mathcal{L}_{\operatorname{Im}(V)}\omega^{T}=0$ for the
initial K\"{a}hler $\omega^{T}$.
\end{proposition}

\begin{remark}
In our application, we only need a compact Fano Sasakian manifold $(M,\eta
,\xi,\Phi,\omega^{T})$ to be quasi-regular. Then Proposition \ref{P77} follows
easily from Cao-Tian-Zhu \cite{ctz} and Tian-Zhu \cite{tzhu1} for the
K\"{a}hler case.
\end{remark}

\begin{corollary}
Let $\mathcal{SR}(n)$ be the space of Sasaki-Ricci solitons on a compact
transverse Fano quasi-regular Sasakian manifold $(M^{2n+1},g^{T},V)$. For any
$(M,g^{T})\in \mathcal{SR}(n),$
\begin{equation}
\mu(g^{T})\geq-A. \label{4}%
\end{equation}

\end{corollary}

\begin{proof}
Let $g^{T}(t)$ be the solution of the normalized Sasaki-Ricci flow with the
initial metric $\widehat{g}^{T}$. Since $g^{T}(t)\in \mathcal{SR}(n),$ it
follows from Proposition \ref{P77} that the Sasaki-Ricci flow with any initial
Sasaki metric in $\mathcal{SR}(n)$ will converge to a Sasaki-Ricci soliton in
the $C^{\infty}$-sense of Sasaki potentials. Moreover
\[%
\begin{array}
[c]{c}%
\lim_{t\rightarrow \infty}\mu(g^{T}(t))=\mu(g^{T}).
\end{array}
\]

Note that $\mu(g^{T}(t))$ is monotonically non-decreasing along the
Sasaki-Ricci flow due to Perelman \cite{p1} and He \cite{he}. Then the
Corollary follows easily from (\ref{3}).
\end{proof}
The Proof of Theorem \ref{T55}:

\begin{proof}
The $C^{\infty}$-convergence on the regular part of $M_{\infty}$ is achieved
by making use of a variant of Perelman's pseudolocality theorem since the
soliton metric is a solution of the Sasaki-Ricci flow. The goal of the rest of
the paper is to show that $M_{\infty}$ is a transverse $\mathbb{Q}$-Fano
Sasakian manifold with foliation Kawamata log terminal singularities. The
Sasaki-Ricci soliton metric $g_{\infty}^{T}$ extends to a transverse
K\"{a}hler current on $M_{\infty}$ with bounded local potential
and\ $V_{\infty}$ extends to a global Hamiltonian holomorphic vector field on
$M_{\infty}$.

Now we can apply the same argument as in Guo-Phong-Song-Sturm \cite{gpss} and
Tian-Zhang \cite{tz2} because the assumption of the uniform bound for the
Futaki invariant is to obtain a uniform lower bound for the $\mu$ functional.
\end{proof}

As a consequence of Proposition \ref{P2026}, it follows from Proposition
\ref{T66} and Compactness Theorem \ref{T55} that the Hamilton-Tian conjecture
holds for the Sasaki-Ricci flow as in Theorem \ref{T99}. That is, under the
Sasaki-Ricci flow (\ref{2025}), we have%
\[%
\begin{array}
[c]{lcl}%
(M,\xi^{i},g^{i}) & \overset{i\rightarrow \infty}{\longrightarrow} &
(M,\xi,g_{0})\\
\downarrow t\rightarrow \infty &  & \downarrow t\rightarrow \infty \\
(M_{\infty},\xi^{i},g^{i}{}_{\infty}) & \overset{i\rightarrow \infty
}{\longrightarrow} & (M_{\infty},\xi,g_{\infty}).
\end{array}
\]
\section{Applications}

As the final consequence, we only deal when the pair is $(M,\emptyset)$. More
precisely, we will show that the gradient Sasaki-Ricci soliton orbifold metric
is a Sasaki-Einstein metric if $M$ is transverse $K$-stable. This is an odd
dimensional counterpart of Yau-Tian-Donaldson conjecture on a compact
$K$-stable K\"{a}hler manifold (\cite{cds1}, \cite{cds2}, \cite{cds3},
\cite{t5}). It can be viewed as a Sasaki analogue of Tian-Zhang's \cite{tz1}
and Chen-Sun-Wang's result \cite{csw} for the K\"{a}hler-Ricci flow.

All transverse quantities on Sasakian manifolds such as transverse Mabuchi
$K$-energy, Sasaki-Futaki invariant, transverse Perelman's $\mathcal{W}%
$-functional can be viewed as their K\"{a}hler counterparts restricted on
basic functions and transverse K\"{a}hler structure. Then all the integrands
are only involved with the transverse K\"{a}hler structure. Hence, under the
Sasaki-Ricci flow, the Reeb vector field and the transverse holomorphic
structure are both invariant, and the metrics are bundle-like. Furthermore,
when one applies integration by parts, the expressions involved behave
essentially the same as in the K\"{a}hler case.

\subsection{Transverse Mabuchi $K$-Energy}

We recall the special case of Log Sasaki-Mabuchi $K$-energy as the transverse
Mabuchi $K$-energy (\cite{cj}) on a compact transverse Fano Sasakian
$(2n+1)$-manifold along any basic transverse K\"{a}hler potential $\phi_{s}$
with $\phi_{0}=\varphi_{1}$ and $\phi_{1}=\varphi_{2}$:%
\begin{equation}%
\begin{array}
[c]{c}%
K_{\eta_{0}}(\varphi_{1},\varphi_{2}):=-\frac{1}{V}\int_{0}^{1}\int
_{M}\overset{\cdot}{\phi}_{s}(R_{\phi_{s}}^{T}-n)\omega_{\phi_{s}}^{n}%
\wedge \eta_{\phi_{s}}\wedge ds
\end{array}
\label{57b}%
\end{equation}
and then also
\[%
\begin{array}
[c]{c}%
K_{\eta_{0}}(\varphi_{1},\varphi_{2}):=-\frac{1}{V}\int_{0}^{1}\int
_{M}\overset{\cdot}{\phi}_{s}(R_{\phi_{s}}^{T}-n)\omega_{\phi_{s}}^{n}%
\wedge \eta_{0}\wedge ds.
\end{array}
\]
It follows easily from the definition that

\begin{lemma}
\label{L50} (\cite{cj}, \cite{fow})

\begin{enumerate}
\item $K_{\eta_{0}}$ is independent of the path $\phi_{t}$, where
$\overset{\cdot}{\phi_{s}}=\frac{d}{dt}\phi.$ Furthermore it satisfies the
$1$-cocycle condition%
\[
K_{\eta_{0}}(\varphi_{1},\varphi_{2})+K_{\eta_{0}}(\varphi_{2},\varphi
_{3})=K_{\eta_{0}}(\varphi_{1},\varphi_{3})
\]
and
\[
K_{\eta_{0}}(\varphi_{1}+C_{1},\varphi_{2}+C_{2})=K_{\eta_{0}}(\varphi
_{1},\varphi_{2}).
\]

\item For a family of transverse K\"{a}hler potentials $\varphi=\varphi_{t}$
in (\ref{2022-1}), we have%
\begin{equation}%
\begin{array}
[c]{c}%
\frac{d}{dt}K_{\eta_{0}}(\varphi)=-\int_{M}||\nabla^{T}u(t)||^{2}\omega
(t)^{n}\wedge \eta_{0}.
\end{array}
\label{31}%
\end{equation}

\end{enumerate}
\end{lemma}

\subsection{The Sasaki-Futaki Invariant}

For the Hamiltonian holomorphic vector field $V,$ $d\pi_{\alpha}(V)$ is a
holomorphic vector field on $V_{\alpha}$ and the complex valued Hamiltonian
function $u_{V}:=\sqrt{-1}\eta(V)$ satisfies
\begin{equation}%
\begin{array}
[c]{c}%
\overline{\partial}_{B}u_{V}=-\frac{\sqrt{-1}}{2}i_{V}d\eta.
\end{array}
\label{59-2}%
\end{equation}
Assume we normalize that $c_{1}^{B}(M)=\frac{1}{2}[d\eta]_{B},$ there exists a
basic function $h_{\omega}$ such that
\[%
\begin{array}
[c]{c}%
Ric^{T}(x,t)-\omega(x,t)=\sqrt{-1}\partial_{B}\overline{\partial}_{B}%
h_{\omega}.
\end{array}
\]

\begin{lemma}
\label{2022-A} (\cite{bgs}, \cite{fow}) The Sasaki-Futaki invariant
\begin{equation}%
\begin{array}
[c]{c}%
f_{M}(V)=\int_{M}V(h_{\omega})\omega^{n}\wedge \eta
\end{array}
\label{59}%
\end{equation}
is only depends on the basic cohomology represented by $d\eta$, and not on the
particular transverse K\"{a}hler metric. It is clear that $f_{M}$ vanishes if
$M$ has a Sasaki-Einstein metric in its basic cohomology class. One also have
the following reformulation :
\begin{equation}%
\begin{array}
[c]{c}%
f_{M}(V)=-n\int_{M}u_{V}(Ric_{\omega}^{T}-\omega)\omega^{n-1}\wedge \eta
=-\int_{M}u_{V}(R_{\omega}^{T}-n)\omega^{n}\wedge \eta.
\end{array}
\label{59-1}%
\end{equation}
\end{lemma}
\subsection{Transverse Perelman's $\mathcal{W}$-Functional}

We recall the transverse Perelman's $\mathcal{W}$-functional on a compact
Sasakian $(2n+1)$-manifold:
\begin{equation}%
\begin{array}
[c]{c}%
\mathcal{W}^{T}(g^{T},f,\tau)=(4\pi \tau)^{-n}\int_{M}[\tau(R^{T}+|\nabla
^{T}f|^{2})+f-2n]e^{-f}\omega^{n}\wedge \eta_{0},
\end{array}
\label{2022-2}%
\end{equation}
for $f\in C_{B}^{\infty}(M;R)$ and $\tau>0$ and define
\[%
\begin{array}
[c]{c}%
\lambda^{T}(g^{T},\tau)=\inf \{ \mathcal{W}^{T}(g^{T},f,\tau):f\in
C_{B}^{\infty}(M,\mathbb{R});\text{ }\int_{M}(4\pi \tau)^{-n}e^{-f}\omega
^{n}\wedge \eta_{0}=1\}.
\end{array}
\]

\begin{lemma}
\label{L51}(\cite{co1})

\begin{enumerate}
\item
\[
-\infty<\lambda^{T}(g^{T},\tau)\leq C.
\]

\item There exists $f_{\tau}\in C_{B}^{\infty}(M,\mathbb{R})$ so that
$\mathcal{W}^{T}(g^{T},f_{\tau},\tau)=\lambda^{T}(g^{T},\tau)$. That is,
\[%
\begin{array}
[c]{c}%
\lambda^{T}(g^{T},\tau)=\inf \{ \mathcal{W}^{T}(g^{T},f,\tau):f\in W_{B}%
^{1,2};\text{ }\int_{M}(4\pi \tau)^{-n}e^{-f}\omega^{n}\wedge \eta_{0}=1\}.
\end{array}
\]

\end{enumerate}
\end{lemma}

Note that $\mathcal{W}$-functional can be expressed as
\[%
\begin{array}
[c]{ccl}%
\mathcal{W}^{T}(g^{T},f) & = & \mathcal{W}^{T}(g^{T},f,\frac{1}{2}%
)+(2\pi)^{-n}(2n)V\\
& = & (2\pi)^{-n}\int_{M}[\frac{1}{2}(R^{T}+|\nabla^{T}f|^{2})+f]e^{-f}%
\omega^{n}\wedge \eta_{0},
\end{array}
\]
where $(g^{T},f)$ satisfies $\int_{M}e^{-f}\omega^{n}\wedge \eta_{0}=V$ and
$\tau=\frac{1}{2}.$ Again
\[%
\begin{array}
[c]{c}%
\mu^{T}(g^{T})=\inf \{ \mathcal{W}^{T}(g^{T},f):f\in C_{B}^{\infty
}(M;\mathbb{R})\text{ \textrm{with} }\int_{M}e^{-f}\omega^{n}\wedge \eta
_{0}=V\}.
\end{array}
\]
It follows from Lemma \ref{L51} (also \cite{he}) that

\begin{corollary}
\label{C51}

\begin{enumerate}
\item $\lambda^{T}(g^{T})$ can be attained by some $f$ which satisfies the
Euler-Lagrange equation :
\begin{equation}%
\begin{array}
[c]{c}%
\Delta_{B}f+f+\frac{1}{2}(R^{T}-|\nabla^{T}f|^{2})=\mu^{T}(g^{T})
\end{array}
\label{51}%
\end{equation}
and
\begin{equation}%
\begin{array}
[c]{c}%
\delta \mu^{T}(g^{T})=-\int_{M}\langle \delta g^{T},Ric^{T}-g^{T}+\nabla
^{T}\overline{\nabla}^{T}f\rangle e^{-f}\omega^{n}\wedge \eta_{0}.
\end{array}
\label{52}%
\end{equation}

\item $g^{T}$ is a critical point of $\mu^{T}(g^{T})$ if and only if $g^{T}$
is a gradient shrinking Sasaki-Ricci soliton%
\begin{equation}
Ric^{T}+\nabla^{T}\overline{\nabla}^{T}f=g^{T} \label{53}%
\end{equation}
where $f$ is a minimizer of $\mathcal{W}^{T}(g^{T},\cdot).$
\end{enumerate}
\end{corollary}

It follows from (\ref{52}) that

\begin{corollary}
Let $(M,\xi,\eta_{0},g_{0})$ be a compact transverse Fano Sasakian
$(2n+1)$-manifold. Then, under the Sasaki-Ricci flow,
\begin{equation}%
\begin{array}
[c]{c}%
\frac{d}{dt}\mu^{T}(g_{t}^{T})=\int_{M}||Ric_{g_{t}^{T}}^{T}-g_{t}^{T}%
+\nabla^{T}\overline{\nabla}^{T}f_{t}||_{g_{t}^{T}}^{2}e^{-f_{t}}\omega
^{n}(g_{t}^{T})\wedge \eta_{0}.
\end{array}
\label{54}%
\end{equation}
Here $f_{t}$ are minimizing solutions of (\ref{51}) associated to metrics
$g_{t}^{T}$ and $\sigma_{t}$ is the family of transverse diffeomorphisms of
$M$ generated by the time-dependent vector field $\frac{1}{2}\nabla_{g_{t}%
^{T}}^{T}f_{t}.$
\end{corollary}

Next, we state some a priori estimates for the minimizing solution $f_{t}$ of
(\ref{51}) under the Sasaki-Ricci \ flow. It can be viewed as their K\"{a}hler
counterparts restricted on basic functions and transverse K\"{a}hler
structure, etc. We refer to the details estimates as in \cite[Proposition
4.2]{tzhu2} and \cite{psswe}.

At first, we can improve the estimates as in Lemma \ref{L31} if the transverse
Mabuchi $K$-energy is bounded from below.

\begin{lemma}
If in addition the transverse Mabuchi $K$-energy is bounded from below on the
space of transverse K\"{a}hler potentials as in Lemma \ref{L31}, then we have

\begin{enumerate}
\item
\[%
\begin{array}
[c]{c}%
\lim_{t\rightarrow \infty}||u(t)||_{C^{0}(M)}=0,
\end{array}
\]

\item
\[%
\begin{array}
[c]{c}%
\lim_{t\rightarrow \infty}||\nabla^{T}u(t)||_{C_{B}^{0}(M,g_{t}^{T})}=0,
\end{array}
\]

\item
\[%
\begin{array}
[c]{c}%
\lim_{t\rightarrow \infty}||\Delta_{B}u(t)||_{C^{0}(M)}=0.
\end{array}
\]
\end{enumerate}
\end{lemma}
As a consequence, we have

\begin{corollary}
\label{C52} If the transverse Mabuchi $K$-energy is bounded from below on the
space of transverse K\"{a}hler potentials as in Lemma \ref{L31}, then there
exists a sequence of $t_{i}\in \lbrack i,i+1]$ such that

\begin{enumerate}
\item
\[%
\begin{array}
[c]{c}%
\lim_{t_{i}\rightarrow \infty}||\Delta_{B}f_{t_{i}}||_{L_{B}^{2}(M,g_{t_{i}%
}^{T})}=0,
\end{array}
\]

\item
\[%
\begin{array}
[c]{c}%
\lim_{t_{i}\rightarrow \infty}||\nabla^{T}f_{t_{i}}||_{L_{B}^{2}(M,g_{t_{i}%
}^{T})}=0,
\end{array}
\]

\item
\[%
\begin{array}
[c]{c}%
\lim_{t_{i}\rightarrow \infty}\int_{M}f_{t_{i}}e^{-f_{t_{i}}}\omega_{g_{t_{i}%
}^{T}}^{n}\wedge \eta_{0}=0.
\end{array}
\]
Moreover, we have
\begin{equation}%
\begin{array}
[c]{c}%
\lim_{t\rightarrow \infty}\mu^{T}(g_{t}^{T})=(2\pi)^{-n}(2n)V=\sup \{ \mu
^{T}(\overline{g}^{T}):\overline{g}^{T}\in c_{1}^{B}(M)\}.
\end{array}
\label{55}%
\end{equation}

\end{enumerate}
\end{corollary}

\subsection{Transverse $K$-Stability}

We first observe that there is a Sasakian analogue of Kodaira embedding
theorem on a compact Sasakian $(2n+1)$-manifold.

\begin{proposition}
\label{PCR}(\cite{rt}, \cite{hlm}, \cite{hhl}) Let $(M,\eta,\xi,\Phi,g)$ be a
compact Sasakian $(2n+1)$-manifold and $(L^{T},h^{T})$ be a basic transverse
holomorphic line bundle over $M$ with the basic Hermitian metric $h^{T}.$ Then
$L^{T}$ is ample if and only if $L^{T}$ is positive.
\end{proposition}

Following the notions as in \cite{dt}, \cite{chlw1}, \cite{t3} and \cite{t5},
we define the Sasaki analogue of a $K$-stable Fano K\"{a}hler manifold on a
compact transverse Fano Sasakian $5$-manifold $(M,\xi,\eta_{0},g_{0})$.

Then by the CR Kodaira embedding theorem (\cite{hhl}, \cite{chlw1}), there
exists an embedding
\[
\Psi:M\rightarrow \Psi(M)\subset(\mathbb{CP}^{N},\omega_{FS})
\]
defined by the basic transverse holomorphic section $\{s_{0},s_{1}%
,\cdots,s_{N}\}$ of $H_{B}^{0}(M,(K_{M}^{T})^{-m})$ which is $\mathbb{R}%
$-equivariant with respect to a weighted diagonal $\mathbb{R}$-actions in
$\mathbb{C}^{N+1}$ with $N=\dim H_{B}^{0}(M,(K_{M}^{T})^{-m})-1$ for a large
positive integer $m.$We define
\[
\mathrm{Diff}^{T}(M)=\{ \sigma \in \mathrm{Diff}(M)\text{\ }|\text{ }%
\sigma_{\ast}\xi=\xi \text{ \textrm{and} }\sigma^{\ast}g^{T}=(\sigma^{\ast
}g)^{T}\}
\]
and
\[
SL^{T}(N+1,\mathbb{C})=SL(N+1,\mathbb{C})\cap \mathrm{Diff}^{T}(M).
\]
Any other basis of $H_{B}^{0}(M,(K_{M}^{T})^{-m})$ gives an embedding of the
form $\sigma^{T}\circ \Psi$ with $\sigma^{T}\in SL^{T}(N+1,\mathbb{C}).$ Now
for any subgroup of the weighted $\mathbb{R}$-action $G_{0}=\{ \sigma
^{T}(t)\}_{t\in \mathbb{C}^{\ast}}$ of $SL^{T}(N+1,\mathbb{C}),$ there is a
unique limiting%
\[%
\begin{array}
[c]{c}%
M_{\infty}=\lim_{t\rightarrow0}\sigma^{T}(t)(M)\subset \mathbb{CP}^{N}.
\end{array}
\]

Now the Sasaki-Futaki invariant can be extended to transverse Fano Sasakian
manifold of foliation log terminal singularities. Let $V$ be a the Hamiltonian
holomorphic vector field whose real part generates the action by $\sigma
^{T}(e^{-s})$. As in our previous results as in \cite{chlw1}, there is a
generalized Sasaki-Futaki invariant defined by $f_{M_{\infty}}(V)$ as in
(\ref{59}) and (\ref{59-1}). In fact, one can introduce the Sasaki analogue of
the $K$-stable on K\"{a}hler manifolds (\cite{t3}, \cite{t5}, \cite{d},
\cite{li}, \cite{tw1}, \cite{dt}) :

\begin{definition}
\label{d61} Let $(M_{\infty},\xi,\eta,g,\omega)$ be a compact transverse Fano
Sasakian manifold at worst foliation log terminal at the singular set of
finite $\mathbb{S}^{1}$-fibres as in Theorem \ref{T99}. We say that a compact
transverse Fano Sasakian $5$-manifold $(M,\xi,\eta_{0},g_{0})$ is transverse
$K$-stable with respect to $(K_{M}^{T})^{-m}$ if the generalized Sasaki-Futaki
invariant
\[
\operatorname{Re}f_{M_{\infty}}(V)\geq0\text{ \  \ }%
\]
for any weighted $\mathbb{C}^{\ast}$-action $G_{0}=\{ \sigma^{T}%
(t)\}_{t\in \mathbb{C}^{\ast}}$ of $SL^{T}(N+1,\mathbb{C})$ and the equality
holds if and only if $M_{\infty}$ is transverse biholomorphic to $M.$ We say
that $M$ is transverse $K$-stable if it is transverse $K$-stable for all large
positive integer $m$.
\end{definition}

\begin{remark}
\label{2025-3}

\begin{enumerate}
\item There are other formulations of the $K$-stability by Donaldson as in
\cite{d} and Paul as in \cite{paul1}. Donaldson's formulation of the
$K$-stability does not require that $M_{\infty}$ be transverse normal.
However, by \cite{lx}, Donaldson's formulation is equivalent to Definition
\ref{d61} if $M_{\infty}$ is a transverse Fano klt normal variety.

\item Collins and Sz\'{e}kelyhidi (\cite{cz2}) called the $K$-stability of
Sasakian manifolds in the sense of its Kaehler cone is $K$-stable by following
the definition of Donaldson.

\item Similar as in \cite{ccllw} for the case of log pair , we say that $M$ is
transverse log $K$-polystable if $M_{\infty}$ is a transverse Fano klt normal variety.
\end{enumerate}
\end{remark}

With Definition \ref{d61} in mind, we are ready to show that the transverse
Mabuchi $K$-energy is bounded from below under the Sasaki-Ricci flow. This is
served as the Sasaki analogue of the K\"{a}hler-Ricci flow due to Tian-Zhang
\cite{tz1}.

\begin{theorem}
Let $(M,\xi,\eta_{0},g_{0})$ be a compact transverse Fano Sasakian manifold of
dimension five. If $M$ is transverse $K$-stable, then the transverse Mabuchi
$K$-energy is bounded from below under the Sasaki-Ricci flow (\ref{2025})%
\begin{equation}
K_{\eta_{0}}(\omega_{0},\omega_{t})\geq-C(g_{0}). \label{54-b}%
\end{equation}

\end{theorem}

\begin{proof}
First, it follows from (\ref{31}) that $K_{\eta_{0}}(\omega_{0},\omega_{t})$
is non-increasing in $t$. So it suffices to show a uniform lower bound of
\begin{equation}
K_{\eta_{0}}(\omega_{0},\omega_{t_{i}})\geq-C. \label{56}%
\end{equation}
Now if $M$ is transverse $K$-stable, we fix an integer $m>0$ sufficiently
large such that $(L^{T})^{-m}$ is very-ample. Then for any orthonormal basic
basis $\{ \sigma_{t_{i},m,k}\}_{k=0}^{N_{m}}$ in $H_{B}^{0}(M,(K_{M}^{T}%
)^{-m})$ with $N_{m}=\dim H_{B}^{0}(M,(K_{M}^{T})^{-m})-1$ at $t_{i}$ , we can
define the $\mathbb{S}^{1}$-equivariant embedding with respective the weighted
$\mathbb{C}^{\ast}$-action in $\mathbb{C}^{N_{m}+1}$
\[
\Psi_{i}:M\rightarrow(\mathbb{CP}^{N_{m}},\omega_{FS})
\]
with the Bergman metric $\overline{\omega}_{t_{i}}:=\overline{\omega}%
_{i}=\frac{1}{m}\Psi_{i}^{\ast}(\omega_{FS})$ so that for any $i\geq1,$, there
exists a $G_{i}\in SL^{T}(N_{m}+1,\mathbb{C})$ such that
\[
\Psi_{i}=G_{i}\circ \Psi_{1}.
\]
At first as in the K\"{a}hler case (\cite{paul1}, \cite{tz1}), the Sasaki
analogue of Mabuchi $K$-energy will have a lower bound on $\overline{\omega
}_{i}$ with \ a fixed $\overline{\omega}_{1}$
\[
K_{\eta_{0}}(\overline{\omega}_{1},\overline{\omega}_{i})\geq-C.
\]
By the $1$-cocycle condition of the transverse Mabuchi $K$-energy,%
\[
K_{\eta_{0}}(\omega_{0},\omega_{i})+K_{\eta_{0}}(\omega_{i},\overline{\omega
}_{i})=K_{\eta_{0}}(\omega_{0},\overline{\omega}_{i})=K_{\eta_{0}}(\omega
_{0},\overline{\omega}_{1})+K_{\eta_{0}}(\overline{\omega}_{1},\overline
{\omega}_{i})
\]
and then
\[
K_{\eta_{0}}(\omega_{0},\omega_{i})+K_{\eta_{0}}(\omega_{i},\overline{\omega
}_{i})\geq-C.
\]
Hence, to show (\ref{56}), we only need to get an upper bound for
\begin{equation}
K_{\eta_{0}}(\omega_{i},\overline{\omega}_{i})\leq C. \label{57}%
\end{equation}

For a fixed $m,$ we define
\[%
\begin{array}
[c]{c}%
\overline{\mathcal{\rho}}_{i}(x):=\frac{1}{m}\mathcal{\rho}_{t_{i}%
,m}(x):=\frac{1}{m}\mathcal{F}_{m}(x,t_{i}).
\end{array}
\]
Here $\mathcal{F}_{m}(x,t_{i})$ as in (\ref{58}). Then
\[%
\begin{array}
[c]{c}%
\omega_{i}=\overline{\omega}_{i}+\sqrt{-1}\partial_{B}\overline{\partial}%
_{B}\overline{\mathcal{\rho}}_{i}.
\end{array}
\]

It follows from \cite{t4} that the transverse Mabuchi $K$-energy has the
following explicit expression%
\begin{equation}%
\begin{array}
[c]{ccl}%
K_{\eta_{0}}(\omega_{i},\overline{\omega}_{i}) & = & \int_{M}\log
\frac{\overline{\omega}_{i}^{n}}{\omega_{i}^{n}}\overline{\omega}_{i}%
^{n}\wedge \eta_{0}+\int_{M}u(\overline{\omega}_{i}^{n}-\omega_{i}^{n}%
)\wedge \eta_{0}\\
&  & -\sqrt{-1}\sum_{k=0}^{n-1}\frac{n-k}{n+1}\int_{M}(\partial_{B}%
\overline{\mathcal{\rho}}_{i}\wedge \overline{\partial}_{B}\overline
{\mathcal{\rho}}_{i}\wedge \omega_{i}^{k}\wedge \overline{\omega}_{i}%
^{n-k-1})\wedge \eta_{0}\\
& \leq & \int_{M}\log \frac{\overline{\omega}_{i}^{n}}{\omega_{i}^{n}}%
\overline{\omega}_{i}^{n}\wedge \eta_{0}+\int_{M}u(\overline{\omega}_{i}%
^{n}-\omega_{i}^{n})\wedge \eta_{0}\\
& \leq & \int_{M}\log \frac{\overline{\omega}_{i}^{n}}{\omega_{i}^{n}}%
\overline{\omega}_{i}^{n}\wedge \eta_{0}+C.
\end{array}
\label{57a}%
\end{equation}
Here $u$ is the transverse Ricci potential under the Sasaki-Ricci flow and we
have used the Perelman estimate that%
\[
|u(t_{i})|\leq C.
\]
On the other hand, it follows from (\ref{d1}) and (\ref{d11}) that%
\[
\overline{\omega}_{i}\leq C\omega_{i}.
\]
Therefore (\ref{57a}) implies (\ref{57}) and then we are done.
\end{proof}

\begin{corollary}
Let $(M,\xi,\eta_{0},g_{0})$ be a compact transverse Fano Sasakian manifold of
dimension five. If $M$ is transverse $K$-stable, then under the Sasaki-Ricci
flow, $M(t)$ converges to a compact transverse Fano Sasakian manifold
$M_{\infty}$ which is isomorphic to $M$ endowed with a smooth Sasaki-Einstein metric.
\end{corollary}

\begin{proof}
Firstly, it follows from Corollary \ref{C52}, (\ref{54-b}) and (\ref{54}) that
$M_{\infty}$ must be Sasaki-Einstein. Moreover, the Lie algebra of all
Hamiltonian holomorphic vector fields is reductive (\cite{dt}, \cite{ber},
\cite{fow}). If $M_{\infty}$ is not equal to $M$, there is a family of the
weighted $\mathbb{C}^{\ast}$-action $\{G(s)\}_{s\in C^{\ast}}\subset
SL^{T}(N_{m}+1,\mathbb{C})$ such that
\[
\Psi_{s}(M)=G(s)\circ \Psi_{1}(M)
\]
converges to the $\mathbb{S}^{1}$-equivariant embedding of $M_{\infty}$ with
respect to the weighted $\mathbb{C}^{\ast}$-action in $(\mathbb{CP}^{N_{m}%
},\omega_{FS}).$ Then the generalized Sasaki-Futaki invariant $f_{M_{\infty}%
}(V)$ of $M_{\infty}$ vanishes.

On the other hand, by the assumption that $M$ is transverse $K$-stable, if
$M_{\infty}$ is not equal to $M$, then
\[
\operatorname{Re}f_{M_{\infty}}(V)>0.
\]
This is a contradiction. Hence $M_{\infty}=M$. Therefore there is a
Sasaki-Einstein metric on $M.$
\end{proof}
\begin{remark}
\label{2025-4}

\begin{enumerate}
\item Note that by continuity method, Collins and Sz\'{e}kelyhidi (\cite{cz2})
showed that a polarized affine variety admits a Ricci-flat K\"{a}hler cone
metric if and only if it is $K$-stable. In particular, the Sasakian manifold
admits a Sasaki-Einstein metric if and only if its K\"{a}hler cone is $K$-stable.

\item On the other hand, instead of $K$-stability on its K\"{a}hler cone, one
can have the so-called transverse $K$-stability on a compact transverse Fano
Sasakian manifold at worst foliation log terminal at the singular set of
$\mathbb{S}^{1}$-fibres of $\{ \xi_{p_{1}},\cdots,\xi_{p_{N}}\}$ as in Remark
(\ref{2025-3}).
\end{enumerate}
\end{remark}

\appendix

\section{ \ }

We first refer to the preliminary notions for a Sasakian manifold in the our
previous papers (\cite{chlw1}, \cite{chlw2}) which included the Sasakian
structure, the leave space and its foliation singularities, basic holomorphic
line bundles and basic divisors over Sasakian manifolds. We also refer to
\cite{bg}, \cite{fow}, and references therein for some details.

\subsection{Sasakian Geometry and Sasakian Structures}

Let $(M,g,\nabla)$ be a Riemannian $(2n+1)$-manifold. $(M,g)$ is called Sasaki
if \ the cone
\[
(C(M),J,\overline{\omega},\overline{g}):=(\mathbf{%
\mathbb{R}
}^{+}\times M\mathbf{,\ }dr^{2}+r^{2}g)
\]
is K\"{a}hler with $\overline{\omega}=\frac{1}{2}i\partial \overline{\partial
}r^{2}$ and
\[
\overline{\eta}=\frac{1}{2}\overline{g}(\xi,\cdot)\text{ \  \  \textrm{and}
\  \ }\overline{\xi}=J(r\frac{\partial}{\partial r}).
\]
The function $\frac{1}{2}r^{2}$ is hence a global K\"{a}hler potential for the
cone metric. As $\left[  r=1\right]  =\{1\} \times M\subset C(M)$, we may
define the Reeb vector field $\xi$ on $M$ by
\[
\xi=J(\frac{\partial}{\partial r}).
\]
and the contact $1$-form $\eta$ on $TM$
\[
\eta=g(\xi,\cdot)
\]
Then $\xi$ is the killing vector field with unit length such that $\eta
(\xi)=1\ $and$\  \ d\eta(\xi,X)=0.$ The tensor field of $type(1,1)$, defined
by
\[
\Phi(Y)=\nabla_{Y}\xi
\]
satisfies the condition%
\[
(\nabla_{X}\Phi)(Y)=g(\xi,Y)X-g(X,Y)\xi
\]
for any pair of vector fields $X$ and $Y$ on $M$. Then such a triple
$(\eta,\xi,\Phi)$ is called a Sasakian structure on a Sasakian manifold
$(M,g).$ Note that the Riemannian curvature satisfying the following
\[
R(X,\xi)Y=g(\xi,Y)X-g(X,Y)\xi
\]
for any pair of vector fields $X$ and $Y$ on $M$. In particular, the sectional
curvature of every section containing $\xi$ equals one.

The first structure theorem on Sasakian manifolds states that

\begin{proposition}
\label{P21}(\cite{ru}, \cite{bg}) Let $(M,\eta,\xi,\Phi,g)$ be a compact
quasi-regular Sasakian manifold of dimension $2n+1$ and $Z$ denote the space
of leaves of the characteristic foliation $\mathcal{F}_{\xi}$ (just as
topological space). Then

\begin{enumerate}
\item $Z$ carries the structure of a Hodge orbifold $\mathcal{Z=}(Z,\Delta)$
with an orbifold K\"{a}hler metric $h$ and K\"{a}hler form $\omega$ which
defines an integral class in $H_{orb}^{2}(Z,\mathbf{Z)}$ in such a way that
$\pi:$ $(M,g,\omega)\rightarrow(Z,h,\omega_{h})$ is an orbifold Riemannian
submersion, and a principal $S^{1}$-orbibundle ($V$-bundle) over $Z.$
Furthermore,it satisfies $\frac{1}{2}d\eta=\pi^{\ast}(\omega_{h}).$The fibers
of $\pi$ are geodesics.

\item $Z$ is also a $Q$-factorial, polarized, normal projective algebraic variety.

\item The orbifold $Z$ is Fano if and only if $Ric_{g}>-2$: In this case $Z$
as a topological space is simply connected; and as an algebraic variety is
uniruled with Kodaira dimension $-\infty$.

\item $(M,\xi,g)$ is Sasaki-Einstein if and only if $(Z,h)$ is
K\"{a}hler-Einstein with scalar curvature $4n(n+1).$

\item If $(M,\eta,\xi,\Phi,g)$ is regular then the orbifold structure is
trivial and $\pi$ is a principal circle bundle over a smooth projective
algebraic variety.

\item As real cohomology classes, there is a relation between the first basic
Chern class and the first orbifold Chern class
\[
c_{1}^{B}(M):=c_{1}(\emph{F}_{\xi})=\pi^{\ast}c_{1}^{orb}(\mathbf{Z}).
\]

\end{enumerate}

Conversely, let $\pi$: $M\rightarrow Z$ be a $\mathbf{S}^{1}$-orbibundle over
a compact Hodge orbifold $(Z,h)$ whose first Chern class is an integral class
defined by $[\omega_{Z}]$, and $\eta$ be a $1$-form with $\frac{1}{2}d\eta
=\pi^{\ast}\omega_{Z}$. Then $(M,\pi^{\ast}h+\eta \otimes \eta)$ is a Sasakian
manifold if all the local uniformizing groups inject into the structure group
$U(1).$
\end{proposition}
\subsection{The Foliation Singularities}

We recall the definition of foliation singularities in a compact quasi-regular
Sasakian $5$-manifold (\cite{chlw2}). Let $(M,\xi,\alpha_{\xi})$ be a compact
quasi-regular Sasakian $5$-manifold and the leave space $Z:=M/\mathbb{S}^{1}$
be an orbifold K\"{a}hler surface. Denote by the quotient map
\[
\pi:M\rightarrow Z
\]
as before, and call such a $5$-manifold $(M,\xi)$ an $\mathbb{S}^{1}%
$-orbibundle. More precisely, $M$ admits a locally free, effective
$\mathbb{S}^{1}$-action
\[%
\begin{array}
[c]{c}%
\alpha_{\xi}:\mathbb{S}^{1}\times M\rightarrow M
\end{array}
\]
such that $\alpha_{\xi}(t)$ is orientation-preserving, for every
$t\in \mathbb{S}^{1}$. Since $M$ is compact, $\alpha_{\xi}$ is proper and the
isotropy group $\Gamma_{p}$ of every point $p\in M$ is finite.

\begin{definition}
The principal orbit type $M_{\operatorname{reg}}$ corresponds to points in
$(M,\xi,\alpha_{\xi})$ with the trivial isotropy group and
$M_{\operatorname{reg}}\rightarrow M_{\operatorname{reg}}/\mathbb{S}^{1}$ is a
principle $\mathbb{S}^{1}$-bundle. Furthermore, the orbit $\mathbb{S}_{p}^{1}$
of a point $p\in M$ is called a regular fiber if $p\in M_{\operatorname{reg}}%
$, and a singular fiber otherwise. In this case, $M_{\mathrm{sing}}%
/\mathbb{S}^{1}\simeq \Sigma^{orb}(Z).$
\end{definition}

In general for a compact quasi-regular Sasakian $5$-manifold, the space of
leaves has at least the codimension two fixed point set of every non-trivial
isotropy subgroup or the codimension one fixed point set of some non-trivial
isotropy subgroup.

\begin{definition}
\label{D21} (\cite{chlw2})

\begin{enumerate}
\item Foliation singularities of type I : Let $(M,\eta,\xi,\Phi,g)$ be a
compact quasi-regular Sasakian $5$-manifold and its leave space $(Z,\emptyset
)$ of the characteristic foliation be well-formed which means its orbifold
singular locus and algebro-geometric singular locus coincide, equivalently $Z$
has no branch divisors.. Then $Z$ is a $\mathbb{Q}$-factorial normal
projective algebraic orbifold surface with isolated singularities of a finite
cyclic quotient of $\mathbb{C}^{2}$. Accordingly, $p\in Z$ is analytically
isomorphic to $p\in Z\simeq(0\in \mathbb{C}^{2})/\mu_{Z_{r}},$where $Z_{r}$ is
a cyclic group of order $r$ and its action on such open affine neighborhood is
defined by
\[
\mu_{Z_{r}}:(z_{1},z_{2})\rightarrow(\zeta^{a}z_{1},\zeta^{b}z_{2}),
\]
where $\zeta$ is a primitive $r$-th root of unity. We denote the cyclic
quotient singularity by $\frac{1}{r}(a,b)$ with $(a,r)=1=(b,r)$. In
particular, the action can be rescaled so that every cyclic quotient
singularity corresponds to a $\frac{1}{r}(1,a)$-point with $(r,a)=1$. It is
klt (Kawamata log terminal) singularities.

\item Foliation singularities of type II : Let $(M,\eta,\xi,\Phi,g)$ be a
compact quasi-regular Sasakian $5$-manifold and its leave space $(Z,\Delta)$
has the codimension one fixed point set of some non-trivial isotropy subgroup.
In this case, the action
\[%
\begin{array}
[c]{c}%
\mu_{Z_{r}}:(z_{1},z_{2})\rightarrow(e^{\frac{2\pi a_{1}i}{r_{1}}}%
z_{1},e^{\frac{2\pi a_{2}i}{r_{2}}}z_{2}),
\end{array}
\]
for some positive integers $r_{1,}$ $r_{2}$ whose least common multiple is
$r$, and $a_{i},$ $i=1,$ $2$ are integers coprime to $r_{i},$ $i=1,$ $2$. Then
the foliation singular set contains some $3$-dimensional Sasakian submanifolds
of $M.$ More precisely, the corresponding singularities in $(M,\eta,\xi
,\Phi,g)$\textbf{\ }is called\textbf{\ }the Hopf\textbf{\ }$\mathbb{S}^{1}%
$\textbf{-}orbibundle over a Riemann surface\textbf{\ }$\Sigma_{h}.$

\item By the second structure theorem (\cite{ru}, \cite{bg}), any Sasakian
structure $(\xi,\eta,\Phi)$ on $(M,g)$ is either quasi-regular or there is a
sequence of quasi-regular Sasakian structures $(M,\xi_{i},\eta_{i},\Phi
_{i},g_{i})$ converging in the compact-open $C^{\infty}$-topology to
$(\xi,\eta,\Phi,g).$ Then all of $(M,\xi_{i},\eta_{i},\Phi_{i},g_{i})$ are the
same type of foliation singularities. Hence we called an irregular Sasakian
manifold $(M,\xi,\eta,\Phi,g)$ of foliation singularities of type I or type II
if $(M,\xi_{i},\eta_{i},\Phi_{i},g_{i})$ is the foliation singularities of
type I or type II, respectively.
\end{enumerate}
\end{definition}
\subsection{The Foliated Normal Coordinate}

Let $(M,\eta,\xi,\Phi,g)$ be a compact Sasakian $(2n+1)$-manifold with
$g(\xi,\xi)=1$ and the integral curves of $\xi$ are geodesics. For any point
$p\in M$, we can construct local coordinates in a neighborhood of $p$ which
are simultaneously foliated and Riemann normal coordinates (\cite{gkn}). That
is, we can find Riemann normal coordinates $\{x,z^{1},z^{2},\cdot \cdot
\cdot,z^{n}\}$ on a neighborhood $U$ of $p$, such that $\frac{\partial
}{\partial x}=\xi$ on $U$. Let $\{U_{\alpha}\}_{\alpha \in A}$ be an open
covering of the Sasakian manifold and $\pi_{\alpha}:U_{\alpha}\rightarrow
V_{\alpha}\subset%
\mathbb{C}
^{n\text{ }}$ be submersions such that
\[
\pi_{\alpha}\circ \pi_{\beta}^{-1}:\pi_{\beta}(U_{\alpha}\cap U_{\beta
})\rightarrow \pi_{\alpha}(U_{\alpha}\cap U_{\beta})
\]
is biholomorphic. On each $V_{\alpha},$ there is a canonical isomorphism
\[
d\pi_{\alpha}:D_{p}\rightarrow T_{\pi_{\alpha}(p)}V_{\alpha}%
\]
for any $p\in U_{\alpha},$ where $D=\ker \xi \subset TM.$ Since $\xi$ generates
isometries, the restriction of the Sasakian metric $g$ to $D$ gives a
well-defined Hermitian metric $g_{\alpha}^{T}$ on $V_{\alpha}.$ This Hermitian
metric in fact is K\"{a}hler. More precisely, let $z^{1},z^{2},\cdot \cdot
\cdot,z^{n}$ be the local holomorphic coordinates on $V_{\alpha}$. We pull
back these to $U_{\alpha}$ and still write the same. Let $x$ be the coordinate
along the leaves with $\xi=\frac{\partial}{\partial x}.$ Then we have the
foliation local coordinate $\{x,z^{1},z^{2},\cdot \cdot \cdot,z^{n}\}$ on
$U_{\alpha}\ $and $(D\otimes%
\mathbb{C}
)$ is spanned by the fields $Z_{j}=\left(  \frac{\partial}{\partial z^{j}%
}+ih_{j}\frac{\partial}{\partial x}\right)  ,\  \  \ j\in \left \{
1,2,...,n\right \}  $ with
\[
\eta=dx-ih_{j}dz^{j}+ih_{\overline{j}}d\overline{z}^{j}%
\]
and its dual frame
\[
\{ \eta,dz^{j},\ j=1,2,\cdot \cdot \cdot,n\}.
\]
Here $h$ is a basic function such that $\frac{\partial h}{\partial x}=0$ and
$h_{j}=\frac{\partial h}{\partial z^{j}},h_{j\overline{l}}=\frac{\partial
^{2}h}{\partial z^{j}\partial \overline{z}^{l}}$ with the foliation normal
coordinate%
\begin{equation}
h_{j}(p)=0,h_{j\overline{l}}(p)=\delta_{j}^{l},dh_{j\overline{l}}(p)=0.
\label{AAA3}%
\end{equation}
Moreover, we have
\[
d\eta(Z_{\alpha},\overline{Z_{\beta}})=d\eta(\frac{\partial}{\partial
z^{\alpha}},\frac{\partial}{\overline{\partial}z^{\beta}}).
\]
Then the K\"{a}hler $2$-form $\omega_{\alpha}^{T}$ of the Hermitian metric
$g_{\alpha}^{T}$ on $V_{\alpha},$ which is the same as the restriction of the
Levi form $d\eta$ to $\widetilde{D_{\alpha}^{n}}$, the slice $\{x=$
\textrm{constant}$\}$ in $U_{\alpha},$ is closed. The collection of K\"{a}hler
metrics $\{g_{\alpha}^{T}\}$ on $\{V_{\alpha}\}$ is so-called a transverse
K\"{a}hler metric. We often refer to $d\eta$ as the K\"{a}hler form of the
transverse K\"{a}hler metric $g^{T}$ in the leaf space $\widetilde{D^{n}}.$

The K\"{a}hler form $d\eta$ on $D$ and the K\"{a}hler metric $g^{T}$ is define
such that $g=g^{T}+\eta \otimes \eta.$ Now in terms of the normal coordinate, we
have%
\[
g^{T}=g_{i\overline{j}}^{T}dz^{i}d\overline{z}^{j}.
\]
Here $g_{i\overline{j}}^{T}=g^{T}(\frac{\partial}{\partial z^{i}}%
,\frac{\partial}{\partial \overline{z}^{j}}).$ The transverse Ricci curvature
$Ric^{T}$ of the Levi-Civita connection $\nabla^{T}$ associated to $g^{T}$ is
defined by $Ric^{T}=Ric+2g^{T}$ and then $R^{T}=R+2n.$ The transverse Ricci
form is defined to be
\[
\rho^{T}=Ric^{T}(\Phi \cdot,\cdot)=-iR_{i\overline{j}}^{T}dz^{i}\wedge
d\overline{z}^{j}%
\]
with
\[
R_{i\overline{j}}^{T}=-\frac{\partial^{2}}{\partial z^{i}\partial \overline
{z}^{j}}\log \det(g_{\alpha \overline{\beta}}^{T})
\]
and it is a closed basic $(1,1)$-form $\rho^{T}=\rho+2d\eta.$

\subsection{The Sasaki-Ricci Flow}

We recall that a $p$-form $\gamma$ on a Sasakian $(2n+1)$-manifold is called
basic if
\[
i(\xi)\gamma=0\text{ \  \  \textrm{and} \  \ }\mathcal{L}_{\xi}\gamma=0.
\]

Let $\Lambda_{B}^{p}$ be the sheaf of germs of basic $p$-forms and
\ $\Omega_{B}^{p}$ be the set of all global sections of $\Lambda_{B}^{p}$. It
is easy to check that $d\gamma$ is basic if $\gamma$ is basic. Set
$d_{B}=d|_{\Omega_{B}^{p}}.$ Then%
\[
d_{B}:=\partial_{B}+\overline{\partial}_{B}:\Omega_{B}^{p}\rightarrow
\Omega_{B}^{p+1}.
\]
with\ $\partial_{B}:\Lambda_{B}^{p,q}\rightarrow \Lambda_{B}^{p+1,q}$ and
$\overline{\partial}_{B}:\Lambda_{B}^{p,q}\rightarrow \Lambda_{B}^{p,q+1}.$
Moreover%
\[
d_{B}d_{B}^{c}=i\partial_{B}\overline{\partial}_{B}\text{ \  \  \textrm{and}
\  \ }d_{B}^{2}=(d_{B}^{c})^{2}=0
\]
for $d_{B}^{c}:=\frac{i}{2}(\overline{\partial}_{B}-\partial_{B}).$ The basic
Laplacian is defined by
\[
\Delta_{B}:=d_{B}d_{B}^{\ast}+d_{B}^{\ast}d_{B}.
\]
Then we have the basic de Rham complex $(\Omega_{B}^{\ast},d_{B})$ and the
basic Dolbeault complex $(\Omega_{B}^{p,\ast},\overline{\partial}_{B})$ and
its cohomology ring $H_{B}^{\ast}(\mathcal{F}_{\xi})\triangleq H_{B}^{\ast
}(M,\mathbf{R})$ of the foliation $\mathcal{F}_{\xi}$ (\cite{eka}). Then we
can define the orbifold cohomology of the leaf space $Z=M/U(1)$ to be this
basic cohomology ring
\[
H_{orb}^{\ast}(Z,\mathbf{R})\triangleq H_{B}^{\ast}(F_{\xi})
\]
and the basic first Chern class $c_{1}^{B}(M)$ by $c_{1}^{B}=[\frac{\rho^{T}%
}{2\pi}]_{B}$. And a transverse K\"{a}hler-Einstein metric(or a Sasaki $\eta
$-Einstein metric) means that it satisfies $[\rho^{T}]_{B}=\varkappa \lbrack
d\eta]_{B}$ for $\varkappa=-1,0,1$, up to a $D$-homothetic deformation.

\begin{example}
Let $(M,\eta,\xi,\Phi,g)$ be a compact Sasakian $(2n+1)$-manifold. If $g^{T}$
is a transverse K\"{a}hler metric on $M,$ then $h_{\alpha}=\det \left(
((g_{i\overline{j}}^{\alpha})^{T})^{-1}\right)  $ on $U_{\alpha}$ defines a
basic Hermitian metric on the transverse canonical bundle $K_{M}^{T}$. The
inverse $(K_{M}^{T})^{-1}$ of $K_{M}^{T}$ is sometimes called the transverse
anti-canonical bundle. Its basic first Chern class $c_{1}^{B}((K_{M}^{T}%
)^{-1})$ is called the basic first Chern class of $M$ and often denoted by
$c_{1}^{B}(M).$Then it follows from the previous result that $c_{1}^{B}(M)$
that $c_{1}^{B}(M)=[\frac{\rho_{\omega}^{T}}{2\pi}]_{B}$ for any transverse
K\"{a}hler metric $\omega$ on a Sasakian manifold $M$.
\end{example}

\begin{definition}
Let $(L,h)$ be a basic transverse holomorphic line bundle over a Sasakian
manifold $(M,\eta,\xi,\Phi,g)$ with the basic Hermitian metric $h$. We say
that $L$ is\textbf{\ }very ample if for any ordered basis $\underline
{s}=(s_{0},...,s_{N})$ of $H_{B}^{0}(M,L)$, the map $i_{\underline{s}%
}:M\rightarrow \mathbf{CP}^{N}$ given by%
\[
i_{\underline{s}}(x)=[s_{0}(x),...,s_{N}(x)]
\]
is well-defined and an embedding which is $\mathbb{R}$-equivariant with
respect to the weighted $\mathbf{C}^{\ast}$action in $\mathbf{C}^{N+1}$ as
long as not all the $s_{i}(x)$ vanish. We say that $L$ is ample if there
exists a positive integer $m_{0}$ such that $L^{m}$ is very ample for all
$m\geq m_{0}.$
\end{definition}

Now we consider the {Type II deformations of Sasakian structures }$(M,\eta
,\xi,\Phi,g)$ as followings :

By fixing the $\xi$ and varying $\eta$, define
\[
\widetilde{\eta}=\eta+d_{B}^{c}\varphi,
\]
for $\varphi \in \Omega_{B}^{0}$. Then
\[
d\widetilde{\eta}=d\eta+i\partial_{B}\overline{\partial}_{B}\varphi \text{
\  \ and \  \ }\widetilde{\omega}=\omega+i\partial_{B}\overline{\partial}%
_{B}\varphi.
\]
Hence we have the same transversal holomorphic foliation but with the new
K\"{a}hler structure on the K\"{a}hler cone $C(M)$ and new contact bundle
$\widetilde{D}$ with%
\[
\widetilde{\omega}=\frac{1}{2}dd^{c}\widetilde{r}^{2},\widetilde
{r}=re^{\varphi}.
\]
Since $r\frac{\partial}{\partial r}=\widetilde{r}\frac{\partial}%
{\partial \widetilde{r}}$ and $\xi+ir\frac{\partial}{\partial r}=\xi-iJ(\xi)$
is a holomorphic vector field on $C(M),$ so we have the same holomorphic
structure. Finally, by the $\partial_{B}\overline{\partial}_{B}$-Lemma in the
basic Hodge decomposition, there is a basic function $F:M\rightarrow
\mathbb{R}
$ such that
\[
\rho^{T}(x,t)-\varkappa d\eta(x,t)=d_{B}d_{B}^{c}F=i\partial_{B}%
\overline{\partial}_{B}F.
\]
Now we focus on finding a new $\eta$-Einstein Sasakian structure
$(M,\xi,\widetilde{\eta},\widetilde{\Phi},\widetilde{g})$ with $\widetilde
{g}^{T}=(g_{i\overline{j}}^{T}+\varphi_{i\overline{j}})dz^{i}d\overline{z}%
^{j}$ such that
\[
\widetilde{\rho}^{T}=\varkappa d\widetilde{\eta}.
\]
Hence $\widetilde{\rho}^{T}-\rho^{T}=\kappa d_{B}d_{B}^{c}\varphi-d_{B}%
d_{B}^{c}F$. It follows that there is a Sasakian analogue of the
Monge-Amp\`{e}re equation for the orbifold version of Calabi-Yau Theorem
\begin{equation}
\frac{\det(g_{\alpha \overline{\beta}}^{T}+\varphi_{\alpha \overline{\beta}}%
)}{\det(g_{\alpha \overline{\beta}}^{T})}=e^{-\kappa \varphi+F}. \label{B}%
\end{equation}

Now we consider the Sasaki-Ricci flow on $M\times \lbrack0,T)$%
\[
\frac{d}{dt}g^{T}(x,t)=-(Ric^{T}(x,t)-\varkappa g^{T}(x,t))
\]
which is equivalent to%
\begin{equation}
\frac{d}{dt}\varphi=\log \det(g_{\alpha \overline{\beta}}^{T}+\varphi
_{\alpha \overline{\beta}})-\log \det(g_{\alpha \overline{\beta}}^{T}%
)+\kappa \varphi-F. \label{C}%
\end{equation}
\subsection{CR Bochner-Kodaira-Nakano Identity}

We consider a CR-holomorphic vector bundle $(E,\overline{\partial}_{b})$ over
a strictly pseudoconvex CR $(2n+1)$-manifold $(M,T^{1,0}(M),\eta)$ and a
unique Chern-Tanaka connection $D$ on $E$ due to N. Tanaka (\cite{ta}). Define
the global $L^{2}$-norm for any $L^{2}$ section $s$ of $E$
\[
\int_{M}||s(x)||^{2}d\mu
\]
where $||s(x)||^{2}=\left \langle s(x),s(x)\right \rangle _{L_{\eta}}$ is the
pointwise Hermitian norm and $d\mu=\eta \wedge(d\eta)^{n}$. The sub-Laplace
Beltrami operator associated to this connection is defined by
\[
\Delta=D^{\ast}D+DD^{\ast}%
\]
where $D^{\ast}$ is the adjoint of $D$ with respect to the $L^{2}$-norm as
above. The Tanaka connection
\[
D=D^{1,0}+D^{0,1}%
\]
have decomposition with
\[
D^{0,1}=\overline{\partial}_{b}.
\]
We define the complex sub-Laplace operators
\[
\Delta^{\prime}=(D^{\prime})^{\ast}D^{\prime}+D^{\prime}(D^{\prime})^{\ast}%
\]
with $D^{\prime}=D^{1,0}$ and
\[
\Delta_{\overline{\partial}_{b}}=\overline{\partial}_{b}^{\ast}\overline
{\partial}_{b}+\overline{\partial}_{b}\overline{\partial}_{b}^{\ast}.
\]
Now we will work on a compact Sasakian manifold $M$. In fact, we consider all
operators on the sheaf of germs of basic $p$-forms $\Lambda_{B}^{p}$ and \ the
set of all global sections $\Omega_{B}^{p}$ of $\Lambda_{B}^{p}$ with the
complex basic Laplacian
\[
\Delta_{\overline{\partial}_{B}}=\overline{\partial}_{B}^{\ast}\overline
{\partial}_{B}+\overline{\partial}_{B}\overline{\partial}_{B}^{\ast}.
\]
Then we have the following CR Bochner-Kodaira-Nakano identity.

\begin{lemma}
\label{L2026} (\cite{chll}) Let $(M,T^{1,0}(M),\eta)$ be a strictly
pseudoconvex Sasakian $(2n+1)$-manifold and $D$ be the Tanaka connection in a
CR-holomorphic vector bundle $E$ over $M$. Then the complex sub-Laplace
operators $\Delta^{\prime}$ and $\Delta_{\overline{\partial}_{B}}$ acting on
$E$-valued basic forms satisfy the identity%
\begin{equation}
\Delta_{\overline{\partial}_{B}}=\Delta^{\prime}+[i\Theta(E),\Lambda],
\label{2020AA}%
\end{equation}
where $\Delta_{\overline{\partial}_{B}}=\overline{\partial}_{B}^{\ast
}\overline{\partial}_{B}+\overline{\partial}_{B}\overline{\partial}_{B}^{\ast
}.$ Here $\Theta(E):=D^{2}s$ is the curvature operator.
\end{lemma}
\subsection{Foliation KLT Singularities}

We recall that a log pair $(Z,\Delta)$ is consisting of a connected compact
projective normal variety $Z$ and an effective $\mathbb{Q}$-divisor $\Delta$
such that $(K_{Z}+\Delta)$ is $\mathbb{Q}$-Cartier. Given a log minimal
resolution $\phi:(Y,D)\rightarrow(Z,\Delta)$ with a nonsingular projective
variety $Y$. Let $(M,\omega_{Z}^{T},\Delta^{T})$ be a compact quasi-regular
Sasakian manifold with type II foliation singulaties. By the first Sasakian
structure theorem, there exists a submersion $\widetilde{\pi}:(\widetilde
{M},D^{T})\rightarrow(Y,D)$ with the regular Sasakian manifold $(\widetilde
{M},D^{T})$ such that the transverse log resolution is defined by%
\begin{equation}
\widetilde{\phi}:(\widetilde{M},D^{T})\rightarrow(M,\Delta^{T}) \label{2026-2}%
\end{equation}
so that it is basic and the following diagram
\[%
\begin{array}
[c]{lcl}%
(\widetilde{M},\widetilde{\omega},D^{T}) & \overset{\widetilde{\phi}%
}{\longrightarrow} & (M,\omega_{Z}^{T},\Delta^{T})\\
\downarrow \widetilde{\pi} & \circlearrowright & \downarrow \pi \\
(Y,\omega,D) & \overset{\phi}{\longrightarrow} & (Z,\omega_{Z},\Delta)
\end{array}
\]
is commutative $\pi \circ \widetilde{\phi}=\phi \circ \widetilde{\pi}$.

In general, by the second Sasakian structure proposition \ref{P22}, the
definition works as well for the irregular case. We refer to \cite{chlw2} for
some details regarding basic holomorphic line bundles $L^{T}$ and basic
divisors $\Delta^{T}$ over Sasakian manifolds, etc.

Then, we define the log pair $(M,\Delta^{T})$ with the foliation klt
(\textrm{log canonical) }singularities as follows :

\begin{definition}
\label{d63} Let $(M,\Delta^{T})$ be a compact Sasakian manifold. We define

\begin{enumerate}
\item We call $(M,\Delta^{T})$ is a log pair with foliation klt singularities
if there exists a unique transverse $\mathbb{Q}$-divisor $D_{\widetilde{M}%
}^{T}$ \ such that
\[
(K_{\widetilde{M}}^{T})^{-1}=\widetilde{\phi}^{\ast}((K_{M}^{T})^{-1}%
-\Delta^{T})+\sum_{0\leq a_{j}<1}a_{j}D_{j}^{T}.
\]
Moreover, $(M,\Delta^{T})$ is a log pair with foliation \textrm{log canonical}
singularities if $0\leq a_{j}\leq1.$

\item $M$ \ admits foliation log terminal singularities when the pair
$(M,\emptyset)$ is foliation klt singularities.

\item $(M,\Delta^{T})$ is called a log Fano Sasakian manifold if $L^{T}%
=(K_{M}^{T})^{-1}-\Delta^{T}$ is a basic ample $\mathbb{Q}$-line bundle. That
is
\[
c_{1}^{B}(M,\Delta^{T})=c_{1}^{B}(M)-c_{1}^{B}(\Delta^{T})>0.
\]

\item A conic K\"{a}hler metric $\omega$ on a compact K\"{a}hler $n$-manifold
$Y$ along the irreducible divisor $D$ with cone angle $2\pi \beta,$ $0<\beta<1$
is asymptotically equivalent to the model metric on $\mathbb{C}^{n}$
\[%
\begin{array}
[c]{c}%
\omega_{\beta}=\sqrt{-1}\frac{dz_{1}\wedge d\overline{z}_{1}}{|z_{1}%
|^{2(1-\beta)}}+\sum_{j=2}^{n}dz_{j}\wedge d\overline{z}_{j}.
\end{array}
\]
It is called conic K\"{a}hler-Einstein
\[
Ric(\omega):=\beta \omega(t)+(1-\beta)[D]
\]
in the sense of currents with $D\thicksim-|K_{Y}|$. That is, locally near a
point $p\in Y,$ there is a holomorphic coordinate system $(V,z_{1}%
,\cdots,z_{n})$ such that $D\cap V=\{z\in V;$ $z_{1}=0\}.$

\item Let $(\widetilde{M},D^{T})$ be a log Fano regular Sasakian manifold. A
conic Sasakian metric $\widetilde{\omega}$ is called a conic Sasaki-Einstein
on $(\widetilde{M},D^{T})$ with the transverse cone angle $2\pi(1-\beta)$
along $D^{T}$ with $\widetilde{\pi}_{\ast}^{-1}(D_{Y})=D^{T}$ if
\[%
\begin{array}
[c]{c}%
\widetilde{Ric}^{T}(\widetilde{\omega})=\widetilde{\omega}+(1-\beta)[D^{T}]
\end{array}
\]
such that $2\pi \beta$ is the cone angle along $D_{Y}.$
\end{enumerate}
\end{definition}


\begin{thebibliography}{99999}
\bibitem[B]{b}D. Barden, \textit{Simply connected five-manifolds}, Ann. of
Math. (2) 82 (1965), 365-385.

\bibitem[Bam]{bam}R. Bamler, Convergence of Ricci flows with bounded scalar
curvature, Ann. of Math. (2) 188 (2018), no. 3, 753--83.

\bibitem[BBEGZ]{bbegz}R. Berman, S. Boucksom, P. Eyssidieux, V. Guedj and A.
Zeriahi, \textit{K\"{a}hler-Einstein metrics and the K\"{a}hler-Ricci flow on
log Fano varieties}, J. Reine Angew. Math. 751 (2019), 27-89.

\bibitem[Bel]{bel}F. A. Belgun, \textit{Normal CR structures on compact }%
$3$\textit{-manifolds}, Math. Z. 238 (2001), no. 3, 441--460.

\bibitem[Ber]{ber}B. Berndtsson, \textit{A Brunn-Minkowski type inequality for
Fano manifolds and some uniqueness theorems in K\"{a}hler geometry}, Invent.
math. (2015) 200:149--200.

\bibitem[BG]{bg}C. P. Boyer and K. Galicki, \textit{Sasaki Geometry}, Oxford
Mathematical Monographs. Oxford University Press, Oxford (2008).

\bibitem[BGS]{bgs}C. P. Boyer , K. Galicki and S. Simanca, \textit{Canonical
Sasakian metrics,} Comm. Math. Phys. 279 (2008), no. 3, 705--733.

\bibitem[BM]{bm}S. Bando and T. Mabuchi,\textit{\ Uniqueness of
Einstein-K\"{a}hler metrics modulo connected group actions}, in: Algebraic
geometry (Sendai 1985), Adv. Stud. Pure Math. 10, North-Holland, Amsterdam
(1987), 11--40.

\bibitem[Bir]{bir}C. Birkar, \textit{Singularities of linear systems and
boundedness of Fano varieties}, Annals of Mathematics 193 (2021), 347-405.

\bibitem[Cr]{cr}C.B. Croke, \textit{Some isoperimetric inequalities and
eigenvalue estimates}, Ann. Sci. Ec. Norm. Super. 13 (1980), 419 435.

\bibitem[CS]{cs}I. A. Cheltsov and K. A. Shramov, \textit{Log-canonical
thresholds for smooth Fano threefolds, with an appendix by J.-P. Demailly},
Uspekhi Mat. Nauk 63 (2008), no. 5 (383), 73--180.

\bibitem[Cao]{cao}H. Cao, \textit{Deformation of K\"{a}hler metrics to
K\"{a}hler-Einstein metrics on compact K\"{a}hler manifolds}, Invent. Math.
81(1985), 359--372.

\bibitem[CC1]{cc1}J. Cheeger and T. H. Colding, \textit{Lower bounds on the
Ricci curvature and the almost rigidity of warped products}, Ann. Math., 144
(1996), 189-237.

\bibitem[CC2]{cc2}J. Cheeger and T. H. Colding, \textit{On the structure of
spaces with Ricci curvature bounded below I}, J. Diff. Geom., 46 (1997), 406-480.

\bibitem[CC3]{cc3}J. Cheeger and T. H. Colding, \textit{On the structure of
spaces with Ricci curvature bounded below II}, J. Diff. Geom., 54 (2000), 13-35.

\bibitem[CCLLW]{ccllw}D.-C. Chang, S.-C. Chang, F. Li, C. Lin and C.-T. Wu,
Yau-Tian-Donaldson Conjecture in a Log Fano Quasi-Regular Sasakian
$5$-Manifold, to appear in Annals of Mathematical Sciences and Applications. arXiv:2406.16430.

\bibitem[CCLW]{cclw}D.-C. Chang, S.-C. Chang, C. Lin and C.-T. Wu,
\textit{Foliation divisorial contraction by the Sasaki-Ricci flow on Sasakian
}$5$\textit{-manifolds}, arXiv: 2203.01736.

\bibitem[CCT]{cct}J. Cheeger, T. H. Colding and G. Tian,\textit{\ On the
singularities of spaces with bounded Ricci curvature}, Geom. Funct. Anal., 12
(2002), 873-914.

\bibitem[CDS1]{cds1}X. Chen, Simon Donaldson, and Song Sun,
\textit{K\"{a}hler-Einstein metrics on Fano manifolds. I,} J. Amer. Math. Soc.
28 (2015), no. 1, 183--197.

\bibitem[CDS2]{cds2}X. Chen, Simon Donaldson, and Song Sun,
\textit{K\"{a}hler-Einstein metrics on Fano manifolds II,} J. Amer. Math. Soc.
28 (2015), no. 1, 199--234.

\bibitem[CDS3]{cds3}X. Chen, Simon Donaldson, and Song Sun,
\textit{K\"{a}hler-Einstein metrics on Fano manifolds III}, J. Amer. Math.
Soc. 28 (2015), no. 1, 235--278.

\bibitem[CHLL]{chll}S.-C. Chang, Y. Han, N. Li and C. Lin, Existence of
nonconstant CR-holomorphic functions of polynomial growth in Sasakian
Manifolds, J. Reine Angew. Math. (Crelle Journal), vol. 2023, no. 802 (2023),
DOI 10.1515/ crelle-2023-0046, pp. 223-254.

\bibitem[CHLW1]{chlw1}S.-C. Chang, Y. Han, C. Lin and C.-T. Wu,\textit{\ }%
$L^{4}$\textit{-bound of the transverse Ricci curvature under the Sasaki-Ricci
flow}, Journal of Mathematical Study, 58(2025), No. 1, 36-59.

\bibitem[CHLW2]{chlw2}S.-C. Chang, Y. Han, C. Lin and C.-T. Wu,
\textit{Convergence of the Sasaki-Ricci flow on Sasakian }$5$%
\textit{-manifolds of general type}, International Journal of Mathematics, 
37(2026), No. 3, 2650020 (47 pages) DOI: 10.1142/S0129167X26500205.

\bibitem[CHLZ]{chlj}S.-C. Chang, Y. Han, C. Lin and J. Zhou, \textit{The
Sasaki-Ricci Flow on Sasaki-Ricci solitons}, preprint.

\bibitem[CHW]{chw}S.-C. Chang, Y. Han and C.-T. Wu, Geometry and Topology of
Gradient Shrinking Sasaki-Ricci Solitons, arXiv: 2508.13495.

\bibitem[CLLQ]{cllq}S.-C. Chang, F. Li, C. Lin and H. Qiu, \textit{Geometry of
shrinking Sasaki-Ricci solitons I: fundamental equations and characterization
of rigidity}, arXiv:2502.16148 ; arXiv:2509.01100.

\bibitem[CJ]{cj}T. Collins and A. Jacob, \textit{On the convergence of the
Sasaki-Ricci flow}, Analysis, complex geometry, and mathematical physics: in
honor of Duong H. Phong, 11--21, Contemp. Math., 644, Amer. Math. Soc.,
Providence, RI, 2015.

\bibitem[CN]{cn}T. H. Colding and A. Naber, \textit{Sharp H\"{o}lder
continuity of tangent cones for spaces with a lower Ricci curvature bound and
applications}, Ann. of Math., 176 (2012), 1173-1229.

\bibitem[Co1]{co1}T. Collins, \textit{The transverse entropy functional and
the Sasaki-Ricci flow}, Trans. AMS., Volume 365, Number 3, March 2013, Pages 1277-1303.

\bibitem[Co2]{co2}T. Collins, \textit{Uniform Sobolev inequality along the
Sasaki-Ricci flow}, J. Geom. Anal. 24 (2014), 1323--1336.

\bibitem[Co3]{co3}T. Collins, \textit{Stability and convergence of the
Sasaki-Ricci flow}, J. reine angew. Math. 716 (2016), 1--27.

\bibitem[CSW]{csw}X, Chen, S. Sun and B. Wang, \textit{K\"{a}hler-Ricci flow,
K\"{a}hler-Einstein metric, and K-stability}, Topol. 22 (2018) 3145-3173.

\bibitem[CTZ]{ctz}H. D. Cao, G. Tian and X. Zhu, \textit{K\"{a}hler-Ricci
solitons on compact complex manifolds with }$C_{1}(M)>0$, GAFA, Geom. funct.
anal. Vol. 15 (2005) 697 -- 719.

\bibitem[CW]{cw}X. X. Chen and B. Wang, Space of Ricci flows (II)---part B:
weak compactness of the flows, J. Differential Geom, 116 (2020), no. 1, 1--123.

\bibitem[CZ1]{cz1}T. Collins and G. Szekelyhidi, \textit{K-semistability for
irregular Sasakian manifolds}, J. Differential Geometry 109 (2018) 81-109.

\bibitem[CZ2]{cz2}T. Collins and G. Szekelyhidi, \textit{Sasaki-Einstein
metrics and K-stability,} Geom. Topol. 23 (2019), no. 3, 1339--1413.

\bibitem[D]{d}S. K. Donaldson, \textit{Scalar curvature and stability of toric
varieties}, J. Differential Geom. 62 (2002), 289-349.

\bibitem[D1]{d1}J.-P. Demailly, \textit{On Tian's invariant and log canonical
thresholds}, appendix to I. Cheltsov and C. Shramov's article : Log-canonical
thresholds of smooth Fano threefolds.

\bibitem[D2]{d2}J.-P. Demailly, $L^{2}$\textit{-vanishing theorems for
positive line bundles and adjunction theory}, Lecture Notes in Math., vol.
1646, Springer, Berlin, 1996, pp. 1-97.

\bibitem[DK]{dk}J.-P. Demailly and J. Koll\'{a}r, \textit{Semi-continuity of
complex singularity exponents and K\"{a}hler-Einstein metrics on Fano
manifolds}, Ann. Ec. Norm. Sup 34 (2001), 525-556.

\bibitem[DP]{dp}J.-P. Demailly and M. Paun, \textit{Numerical characterization
of the K\"{a}hler cone of a compact K\"{a}hler manifold}, Ann. of Math., 159
(2004), no. 3, 1247--1274.

\bibitem[DS]{ds}S. Donaldson and S. Sun, \textit{Gromov-Hausdorff limits of
K\"{a}hler manifolds and algebraic geometry}, Acta Math. 213(1) (2014) 63--106.

\bibitem[DT]{dt}W. Ding and G. Tian, \textit{K\"{a}hler-Einstein metrics and
the generalized Futaki invariants}, Invent. Math., 110 (1992), 315-335.

\bibitem[EKA]{eka}A. El Kacimi-Alaoui, \textit{Operateurs transversalement
elliptiques sur un feuilletage riemannien et applications}, Compos. Math. 79
(1990) 57--106.

\bibitem[F]{f}A. Futaki, \textit{An obstruction to the existence of Einstein
K\"{a}hler metrics,} Invent. Math. 73 (1983), 437-443.

\bibitem[FOW]{fow}A. Futaki, H. Ono and G.Wang, \textit{Transverse K\"{a}hler
geometry of Sasaki manifolds and toric Sasaki-Einstein manifolds}, J.
Differential Geom. 83 (2009) 585--635.

\bibitem[G]{gei}H. Geiges, \textit{Normal contact structures on 3-manifolds},
Tohoku Math. J. 49 (1997), 415-422.

\bibitem[GMSW]{gmsw}J. P. Gauntlett, D. Martelli, J. Sparks and D. Waldram,
\textit{Sasaki-Einstein Metrics on }$\mathbb{S}^{2}\times \mathbb{S}^{3}$, Adv.
Theor. Math. Phys. 8 (2004), 711--734.

\bibitem[GKN]{gkn}M. Godlinski, W. Kopczynski and P. Nurowski, \textit{Locally
Sasakian manifolds}, Classical Quantum Gravity 17 (2000) L105--L115.

\bibitem[GPSS]{gpss}Bin Guo , Duong H. Phong , Jian Song\ and Jacob Sturm,
\textit{Compactness of K\"{a}hler-Ricci solitons on Fano manifolds}, Pure
Appl. Math. Q. 18 (2022), no. 1, 305-316.

\bibitem[H]{h}L. Hormander, \textit{An introduction to complex analysis in
several variables}, Van Nostrand, Princeton, 1973.

\bibitem[H1]{h1}R. S. Hamilton, \textit{The Ricci flow on surfaces}, Math, and
General Relativity, Contemporary Math. 71 (1988), 237-262.

\bibitem[H2]{h2}R. S. Hamilton, T\textit{hree-manifolds with positive Ricci
curvature}, J. Differential Geom. 17 (1982), no. 2, 255--306.

\bibitem[H3]{h3}R.S. Hamilton, \textit{The formation of singularities in the
Ricci flow}, in Surveys in differential geometry, Vol. II (Cambridge, MA,
1993), 7-136, Int. Press, Cambridge, MA, 1995.

\bibitem[He]{he}W. He, \textit{The Sasaki-Ricci flow and compact Sasaki
manifolds of positive transverse holomorphic bisectional curvature}, J. Geom.
Anal. 23 (2013), 1876-931.

\bibitem[HHL]{hhl}H. Herrmann, C.-Y. Hsiao and X. Li, \textit{Szeg\"{o}
kernels and equivariant embedding theorems for CR manifolds}, Math. Res. Lett.
29 (2022), no. 1, 193-246.

\bibitem[HLM]{hlm}C.-Y. Hsiao, X. Li and G. Marinescu, \textit{Equivariant
Kodaira embedding for CR manifolds with circle action}, Michigan Math. J. 70
(2021), no. 1, 55--113.

\bibitem[HS]{hs}W. He and S. Sun, \textit{Frankel conjecture and Sasaki
geometry}, Advances in Mathematics, 291 (2016), 912?60.

\bibitem[JST]{jst}W. Jian, J. Song and G. Tian, Finite time singularities of
the Kaehler-Ricci flow, arXiv:2310.07945v2.

\bibitem[K1]{k1}J. Koll\'{a}r, \textit{Singularities of pairs}, Algebraic
geometry, Santa Cruz 1995, Proc. Sympos. Pure Math., vol. 62, Amer. Math.
Soc., Providence, RI (1997) 221-287.

\bibitem[K2]{k2}J. Koll\'{a}r,\textit{\ Einstein metrics on connected sums of
}$\mathbb{S}^{2}\times \mathbb{S}^{3}$, J. Differential Geom. 75 (2007), no. 2, 259--272.

\bibitem[K3]{k3}J. Koll\'{a}r, \textit{Einstein metrics on five-dimensional
Seifert bundles}, J. Geom. Anal. 15 (2005), no. 3, 445--476.

\bibitem[Li]{li}C. Li, \textit{Remarks on logarithmic }$K$\textit{-stability},
Communications in Contemporary Mathematics, 17 (2014), no. 2, 1450020, 1-17.

\bibitem[LX]{lx}C. Li and C. Xu, \textit{Special test configuration and
K-stability of Fano varieties}, Ann. of Math. (2) 180 (2014), no. 1, 197--232.

\bibitem[L]{l}J. Liu, \textit{The generalized K\"{a}hler Ricci flow}, J. Math.
Anal. Appl. 408 (2013) 751--761.

\bibitem[LZ]{lz}J. Liu and X. Zhang, \textit{The conical K\"{a}hler-Ricci flow
on Fano manifolds}, Advances in Mathematics, 307 (2017), 1324--1371.

\bibitem[M]{m}T. Mabuchi, \textit{K-energy maps integrating Futaki
invariants}, Tohoku Math. J., 38, 245-257 (1986).

\bibitem[MSY]{msy}Dario Martelli, James Sparks and Shing-Tung Yau,
\textit{Sasaki--Einstein manifolds and volume minimisation}, Communications in
Mathematical Physics, 280 (2008), 611-673.

\bibitem[Mo]{mo}N. Mok, \textit{The uniformization theorem for compact
K\"{a}hler manifolds of nonnegative holomorphic bisectional curvature}, J.
Differ. Geom. 27(2) (1988), 179-214.

\bibitem[Na]{na}A. M. Nadel, \textit{Multiplier ideal sheaves and
K\"{a}hler-Einstein metrics of positive scalar curvature}, Ann. of Math. (2)
132 (1990), no. 3, 549-596.

\bibitem[NT]{nt}S. Nishikawa and P. Tondeur, \textit{Transversal infinitesimal
automorphisms for harmonic K\"{a}hler foliation}, Tohoku Math. J., 40(1988), 599-611.

\bibitem[P1]{p1}G. Perelman, \textit{The entropy formula for the Ricci flow
and its geometric applications}, preprint, arXiv: math.DG/0211159.

\bibitem[P2]{p2}G. Perelman, \textit{Ricci flow with surgery on
three-manifolds}, preprint, arXiv: math.DG/0303109.

\bibitem[P3]{p3}G. Perelman, \textit{Finite extinction time for the solutions
to the Ricci flow on certain three-manifolds}, preprint, arXiv: math.DG/0307245.

\bibitem[Paul1]{paul1}S. T. Paul, \textit{Hyperdiscriminant polytopes, Chow
polytopes, and Mabuchi energy asymptotics}, Ann. Math., 175 (2012), 255-296.

\bibitem[Paul2]{paul2}S. T. Paul, \textit{A numerical criterion for }%
$K$\textit{-energy maps of algebraic manifolds}, arXiv:1210.0924v1.

\bibitem[Pe]{pe}P. Petersen, \textit{Convergence theorems in Riemannian
geometry}, in \textquotedblright Comparison Geometry\textquotedblright \ edited
by K. Grove and P. Petersen, MSRI Publications, vol 30 (1997), Cambridge Univ.
Press, 167-202.

\bibitem[PSS]{pss}D. H. Phong, J. Song and J. Sturm, \textit{Degeneration of
K\"{a}hler-Ricci solitons on Fano manifolds}, Acta Math. No. 52 (2015), 29--43.

\bibitem[PSSWe]{psswe}D. H. Phong, J. Song, J. Sturm and B. Weinkove,
\textit{The K\"{a}hler-Ricci flow and }$\overline{\partial}$\textit{-operator
on vector fields}, J. Differ. Geom., 81 (2009), 631-647.

\bibitem[PW1]{pw1}P. Petersen and G.F. Wei,\textit{\ Relative volume
comparison with integral curvature bounds}, Geom. Funct. Anal., 7 (1997), 1031-1045.

\bibitem[PW2]{pw2}P. Petersen and G.F. Wei, \textit{Analysis and geometry on
manifolds with integral Ricci curvature bounds,\ II}, Trans. AMS., 353 (2001), 457-478.

\bibitem[RT]{rt}J. Ross and R.P. Thomas, \textit{Weighted projective
embeddings, stability of orbifolds and constant scalar curvature K\"{a}hler
metrics}, JDG 88 (2011), no. 1, 109-160.

\bibitem[Ru]{ru}P. Rukimbira, \textit{Chern-Hamilton's conjecture and
K-contactness}, Houston J. Math. 21 (1995), no. 4, 709-718.

\bibitem[S]{s}S. Smale, \textit{On the structure of 5-manifolds,} Ann. of
Math. (2) 75 (1962), 38-46.

\bibitem[Shi]{shi}W.-X. Shi, \textit{Ricci deformation of the metric on
complete noncompact Riemannian manifolds}, J. Diff. Geom., 30 (1989), 303-394.

\bibitem[Siu]{siu}Y.-T. Siu, \textit{The existence of K\"{a}hler-Einstein
metrics on manifolds with positive anticanonical line bundle and a suitable
finite symmetry group}, Ann. of Math. 127 (1988), 585--627.

\bibitem[Sp]{sp}James Sparks, \textit{Sasaki-Einstein Manifolds}, Surveys in
Differential Geometry 16 (2011), 265-324.

\bibitem[ST]{st}J. Song and G. Tian, \textit{The K\"{a}hler-Ricci flow through
singularities}, Invent. Math. 207 (2017), no. 2, 519--595.

\bibitem[SW]{sw}J. Song and B. Weinkove, \textit{Contracting exceptional
divisors by the K\"{a}hler-Ricci flow}, Duke Math. J. 162 (2013), no. 2, 367--415.

\bibitem[SWZ]{swz}K. Smoczyk, G. Wang and Y. Zhang, \textit{The Sasaki-Ricci
flow,} Internat. J. Math. 21 (2010), no. 7, 951--969.

\bibitem[SZ]{sz}Y. Shi and X. Zhu, \textit{K\"{a}hler-Ricci solitons on toric
Fano orbifolds}, Math. Z. (2012) 271:1241--1251.

\bibitem[Ta]{ta}N. Tanaka, \textit{A Differential Geometric Study on Strongly
Pseudo-Convex Manifold}, Kinokuniya, Tokyo, 1975.

\bibitem[T1]{t1}G. Tian, \textit{On Calabi's conjecture for complex surfaces
with positive first Chern class}, Invent. Math., 101, (1990), 101-172.

\bibitem[T2]{t2}G. Tian, \textit{Partial }$C^{0}$\textit{-estimates for
K\"{a}hler-Einstein metrics}, Commun. Math. Stat., 1 (2013), 105-113.

\bibitem[T3]{t3}G. Tian,\textit{\ K\"{a}hler-Einstein metrics with positive
scalar curvature}, Invent. Math. 137 (1997), no. 1, 1-37.

\bibitem[T4]{t4}G. Tian, \textit{Canonical Metrics in K\"{a}hler Geometry},
Lectures in Mathematics ETH Z\={u}rich, Birkh\u{u}user Verlag, Basel, 2000.

\bibitem[T5]{t5}G. Tian, \textit{K-stability and K\"{a}hler-Einstein metrics},
Comm. Pure Appl. Math. 68 (2015), no. 7, 1085--1156. Corrigendum: Comm. Pure
Appl. Math. 68 (2015), no. 11, 2082--2083.

\bibitem[TZ1]{tz1}G. Tian and Z. Zhang, \textit{Regularity of K\"{a}ler-Ricci
flow on Fano manifolds}, Acta Math. 216 (2016) No. 1, 127-176.

\bibitem[TZ2]{tz2}G. Tian and Z. Zhang, \textit{Degeneration of
K\"{a}hler-Ricci solitons}, Int. Math. Res. Not. IMRN 2012, no. 5, 957--985.

\bibitem[TZhu1]{tzhu1}G. Tian and X.-H. Zhu, \textit{Convergence of
K\"{a}hler-Ricci flow}, J. Amer. Math. Soc. 20 (2007), 675-699.

\bibitem[TZhu2]{tzhu2}G. Tian and X.-H. Zhu, \textit{Convergence of
K\"{a}hler-Ricci flow on Fano manifolds II}, J. Reine Angew. Math.. 678
(2013), 223-245.

\bibitem[TW1]{tw1}G. Tian and F. Wang, \textit{On the existence of conic
K\"{a}hler-Einstein metrics}, Adv. Math. 375 (2020), 107413, 42 pp.

\bibitem[TW2]{tw2}G. Tian and F. Wang, \textit{Cheeger-Colding-Tian theory for
conic K\"{a}hler-Einstein metrics}, J. Geom. Anal. 31 (2021), no. 2, 1471--1509.

\bibitem[TZZZZ]{tzzzz}G. Tian, Q. Zhang, Z.-L. Zhang, M. Zhu and X.-H. Zhu,
Laplacian comparison on Kaehler-Ricci flow and convergence, arXiv:2509.14820v2.

\bibitem[WZ]{wz}X. Wang and X. Zhu, \textit{K\"{a}hler-Ricci solitons on toric
Fano manifolds with positive first Chern class}. Adv. Math. 188 (2004), 87--103.

\bibitem[Y1]{y1}S.-T. Yau, \textit{On the Ricci curvature of a compact
K\"{a}hler manifold and the complex Monge-Amp\`{e}re equation},\textit{\ I},
Comm. Pure. Appl. Math. 31 (1978), 339-411.

\bibitem[Y2]{y2}S.-T. Yau, \textit{Open problems in geometry}, Proc. Symp.
Pure Math. 54 (1993), 1-18.

\bibitem[Z]{z}Z. Zhang, \textit{Degeneration of shrinking Ricci solitons},
IMRN, no. 21 (2010), 4137--4158.

\bibitem[Zh1]{zh1}X. Zhang, \textit{Some invariants in Sasakian geometry},
International Mathematics Research Notices, Volume 2011, Issue 15 (2011), 3335--3367.

\bibitem[Zh2]{zh2}X. Zhang, \textit{Energy properness and Sasakian-Einstein
metrics}, Commun. Math. Phys. 306 (2011), 229--260.
\end{thebibliography}
\end{document}